\documentclass{amsart}

\usepackage{enumitem}
\usepackage{amsthm,amssymb} 
\usepackage[leqno]{amsmath}
\usepackage{tikz}

\usepackage{hyperref}  
\hypersetup{%
  bookmarksnumbered=true,%
  bookmarks=true,%
  colorlinks=true,%
  linkcolor=blue,%
  citecolor=blue,%
  filecolor=blue,%
  menucolor=blue,%
  pagecolor=blue,%
  urlcolor=blue,%
  pdfnewwindow=true,%
  pdfstartview=FitBH}

 \usepackage[all]{xy}
\usepackage{pinlabel}
\usepackage{enumerate}
\usepackage{framed}
\usepackage{datetime}

\DeclareMathAlphabet\EuScript{U}{eus}{m}{n}
\SetMathAlphabet\EuScript{bold}{U}{eus}{b}{n}

\def\CP{\mathbb{CP}}

\def\hookar{\ar@{^{(}->}}
\def\spec{\mathsf{Spec}}

\def\nov{r}

\def\fuk{\EuF}
\def\bc{\mathsf{bc}}

\newcommand{\iii}{\mathbf{i}}

\newcommand{\fg}{\mathfrak{g}}
\newcommand{\fm}{\mathfrak{m}}
\newcommand{\fh}{\mathfrak{h}}

\newcommand{\power}[1]{\left[\!\left[ #1 \right]\!\right]}
\newcommand{\laurents}[1]{\left(\!\left( #1 \right)\!\right)}
\newcommand{\wt}[1]{\tilde{#1}}
\def\cM{\mathcal{M}}

\def\mkah{\cM_{K\ddot{a}h}}
\def\mbarkah{\overline{\cM}_{K\ddot{a}h}}
\def\mbarhatkah{\hat{\overline{\cM}}_{K\ddot{a}h}}
\def\mhatkah{\hat{\cM}_{K\ddot{a}h}}
\def\mbarcpx{\overline{\cM}_{cpx}}
\def\mbarhatcpx{\hat{\overline{\cM}}_{cpx}}
\def\mhatcpx{\hat{\cM}_{cpx}}
\def\mcpx{\cM_{cpx}}

\def\and{\, \& \,}

\def\cM{\mathcal{M}}

\def\Dcoh{\mathsf{DCoh}}
\def\Dfuk{\mathsf{DFuk}}
\def\fuk{\mathsf{Fuk}}
\def\wfuk{\mathsf{WFuk}}
\def\Dwfuk{\mathsf{DWFuk}}
\def\cA{\mathsf{A}}
\def\cB{\mathsf{B}}
\def\C{\mathbb{C}}
\def\R{\mathbb{R}}
\def\Z{\mathbb{Z}}
\def\End{\mathsf{End}}
\def\EuX{\EuScript{X}}
\def\EuC{\EuScript{C}}
\def\EuO{\EuScript{O}}
\def\id{\mathrm{id}}
\def\Hom{\mathsf{Hom}}
\def\spec{\mathsf{Spec}}
\def\perf{\mathsf{Perf}}

\newtheorem{thm}{Theorem}[section]
\numberwithin{equation}{section}

\newtheorem{strat}{Strategy}
\newtheorem{rmk}[thm]{Remark}
\newtheorem{example}[thm]{Example}
\newtheorem{defn}[thm]{Definition}
\newtheorem{conj}[thm]{Conjecture}

\begin{document}

\title{Versality in mirror symmetry}

\author{Nick Sheridan}

\begin{abstract}
One of the attractions of homological mirror symmetry is that it not only implies the previous predictions of mirror symmetry (e.g., curve counts on the quintic), but it should in some sense be `less of a coincidence' than they are and therefore easier to prove. 
In this survey we explain how Seidel's approach to mirror symmetry via versality at the large volume/large complex structure limit makes this idea precise.
\end{abstract}

\maketitle

\tableofcontents

\section{Introduction}
\label{sec:int}

\subsection{Mirror symmetry for Calabi--Yaus}

Let us introduce the geometry that we will study. 
We will call a complex manifold $X$ \emph{Calabi--Yau} if $c_1(X) = 0$. 
A \emph{complexified K\"ahler form} is a closed form $\omega^\C = B+\iii\omega \in \Omega^{1,1}(X;\C)$, such that $\omega$ is a K\"ahler form. 
We define a \emph{Calabi--Yau K\"ahler manifold} to be a complex Calabi--Yau 
equipped with a complexified K\"ahler form.

We will consider invariants of Calabi--Yau K\"ahler manifolds, which only depend on part of the structure:
\begin{itemize}
\item The \emph{$B$-model} refers to invariants of the complex manifold $X$.
\item The \emph{$A$-model} refers to invariants of the symplectic manifold $(X,\omega)$ together with the `$B$-field' $B \in \Omega^2(X;\R)$.\footnote{The `$A$-model/$B$-model' terminology is inherited from \cite{Witten1992}.}
\end{itemize}

The \emph{mirror symmetry conjecture} posits the existence of pairs $(X,X^\circ)$ of Calabi--Yau K\"ahler manifolds whose $A$- and $B$-models are swapped:
\begin{equation}
\label{eqn:meta_ms}
\xymatrix{A(X) \ar@{<->}[dr] & A(X^\circ)\ar@{<->}[dl] \\
B(X) & B(X^\circ).}
\end{equation}
We're being vague because the definition of the $A$- and $B$-models, and the corresponding notion of equivalence implied by the arrows `$\leftrightarrow$' above, is different in different contexts. 
The `mirror' in the name refers to one of the characteristic features of mirror pairs, namely that the Hodge diamond of $X^\circ$ is a reflection of that of $X$ in an axis on a $45^\circ$ angle: $h^{p,q}(X^\circ) = h^{n-p,q}(X)$.\footnote{An early hint of mirror symmetry came when this peculiar symmetry showed up in a census of Hodge numbers of Calabi--Yau hypersurfaces in weighted projective spaces: see the celebrated Figure 1 in \cite{Candelas1990}.}

\begin{rmk}
\label{rmk:wiigf}
Given a result or structure in algebraic geometry, we can hope that it has a `mirror' in symplectic geometry (and vice-versa). 
If we had a sufficiently precise understanding of mirror symmetry the mirror result would be true `by mirror symmetry'. 
So far we usually do not have such a precise understanding, and to date mirror symmetry has been most useful as a motivating principle: we search for the mirror result or structure within its own world, and deal with it there without reference to mirror symmetry (see Remark \ref{rmk:hmsapps}).\footnote{In fact according to Givental \cite{Givental1999}, ``the mirror conjecture resembles a universal `physical' principle (like the Energy Conservation Law which is to be modified every time it conflicts with experimental data) rather than a precise mathematical conjecture which is to be proved or disproved.''} However, our focus in this paper is on proving precise versions of the mirror conjecture.
\end{rmk}

We can study the two arrows in \eqref{eqn:meta_ms} independently, so it is simplest to focus our attention on one of them. 
For the rest of the paper we will study the equivalence $A(X) \leftrightarrow B(X^\circ)$.

In order for the conjecture to have content, we need some interesting examples of mirror pairs $(X,X^\circ)$, and also to say what $A(X)$ and $B(X^\circ)$ are. 
We will start by addressing the first question. 
In this paper we will mainly be interested in compact Calabi--Yau mirror pairs $(X,X^\circ)$.

\begin{example}
\label{eg:ell_curve}
The mirror to an elliptic curve is another elliptic curve. 
\end{example}

\begin{example}
\label{eg:quartic}
Let $X$ be a \emph{quartic surface}, i.e., a smooth degree-$4$ hypersurface in $\CP^3$. The family of \emph{mirror quartics} $X^\circ_q$ is obtained in three stages: first, we consider hypersurfaces of the form 
\[\{z_1z_2z_3z_4 = q (z_1^4+z_2^4+z_3^4+z_4^4)\} \subset \CP^3;\]
second we take the quotient by $G := \ker((\Z/4)^4/(1,1,1,1) \xrightarrow{+} \Z/4)$, acting by multiplication of the homogeneous coordinates by powers of $i$; third, we resolve the six $A_3$ singularities of this quotient, to obtain the smooth complex manifold $X^\circ_q$.
\end{example}

\begin{example}
\label{eg:quintic}
Let $X$ be a \emph{quintic threefold}, i.e., a smooth degree-$5$ hypersurface in $\CP^4$. 
The construction of the family of \emph{mirror quintics} $X^\circ_q$ is analogous to that of the family of mirror quartics, replacing `$4$' with `$5$' everywhere (the resolution procedure is more complicated though: there are $100$ exceptional divisors). 
\end{example}

Batyrev gave the following general construction of mirror Calabi--Yau hypersurfaces in toric varieties \cite{Batyrev1993}:

\begin{example}
\label{eg:Batyrev}
Let $\Delta \subset \R^n$ be a lattice polytope. 
There is a polarized toric variety $(V,\mathcal{L})$ such that the lattice points in $\Delta$ are in bijection with a basis for the global sections of $\mathcal{L}$ (see \cite[\S 3.4]{Fulton1993}).
If $\Delta$ contains the origin then we can define the \emph{dual} polytope $\Delta^\circ$: it is the set of points in $\R^n$ whose inner product with every point in $\Delta$ is $\ge -1$. 
If $\Delta^\circ$ is also a \emph{lattice} polytope, we say $\Delta$ is \emph{reflexive}. 
The dual polytope $\Delta^\circ$ corresponds to a dual polarized toric variety $(V^\circ,\mathcal{L}^\circ)$. 
If $\Delta$ is reflexive, then the hyperplane sections $X \subset V$ and $X^\circ \subset V^\circ$ (more precisely, their crepant resolutions, if they are singular) are Calabi--Yau, and conjectured to be mirror. 
\end{example}

Example \ref{eg:quartic} (respectively, Example \ref{eg:quintic}) can be obtained from Batyrev's construction by taking $\Delta$ to be the lattice polytope corresponding to the polarization $\mathcal{L} = \mathcal{O}(4)$ of $\CP^3$ (respectively, the polarization $\mathcal{L} = \mathcal{O}(5)$ of $\CP^4$).

We do not have a complete conjectural list of mirror pairs. The most general conjectural construction is that of Gross and Siebert \cite{Gross2013}. 

\begin{rmk}
Other influential mirror constructions that we've omitted include \cite{Greene1990,Borisov1993,Dolgachev1996,Batyrev1994}. 
We've also omitted the string-theoretic history of mirror symmetry: see, e.g., \cite{Hori2003}.
\end{rmk}

\subsection{Moduli spaces}
\label{subsec:moduli}

Having described some examples of mirror pairs $(X,X^\circ)$, the next step is to explain what $A(X)$ and $B(X^\circ)$ are. 
However before we get there, we will discuss the parameters on which these invariants depend.

Let $X$ be a Calabi--Yau K\"ahler manifold.  
The \emph{K\"ahler moduli space} $\mkah(X)$ is the `moduli space of complexified K\"ahler forms on $X$', with tangent space $T\mkah(X) \cong H^{1,1}(X;\C)$.
Similarly the \emph{complex moduli space} $\mcpx(X)$ is the `moduli space of complex structures on $X$', with tangent space $T\mcpx(X) \cong H^1(X,TX)$.
A slightly more precise version of mirror symmetry says that there should exist an isomorphism
\begin{equation}
\label{eqn:mirr_map}
\psi: \mkah(X) \xrightarrow{\sim} \mcpx(X^\circ)
\end{equation}
called the `mirror map', and an equivalence
\begin{equation}
 A\left(X,\omega^\C_p\right) \leftrightarrow B\left(X^\circ_{\psi(p)}\right).
\end{equation}

\begin{rmk}
Note that the dimension of $\mkah(X)$ is $h^{1,1}(X)$, while the dimension of $\mcpx(X^\circ)$ is $h^{n-1,1}(X^\circ)$ by Serre duality. 
The existence of the mirror map implies these dimensions should be the same, which is consistent with the Hodge diamond flip $h^{p,q}(X) = h^{n-p,q}(X^\circ)$. 
\end{rmk}

\begin{example}
\label{eg:ell_curve_ms}
If $X^\circ$ is an elliptic curve, $\mcpx(X^\circ)$ is the well-known moduli space $\mathcal{M}_{1,1}$ of complex elliptic curves. 
It is isomorphic to the quotient of the upper half-plane by $\text{SL}_2(\Z)$, and is a genus-zero curve with one cusp and two orbifold points.
\end{example}

\begin{example}
\label{eg:quartic_ms}
We saw in Example \ref{eg:quartic} that the mirror quartic $X^\circ_q$ depended on a complex parameter $q \in \CP^1$. 
Deforming $q$ does not give all deformations of the complex structure of $X^\circ$ (which is $20$-dimensional as $X^\circ$ is a $K3$ surface); rather it gives a one-dimensional subspace which we denote by $\mcpx^{pol}(X^\circ) \subset \mcpx(X^\circ)$. 
It is not hard to see that $X^\circ_q \cong X^\circ_{\iii q}$, and in fact one has 
\[\mcpx^{pol}(X^\circ) \cong \left(\CP^1 \setminus \{q=0,q^4=1/4^4\}\right)/(\Z/4) \]
 (we remove the fibres over $q=0$ and $q^4=1/4^4$ because $X^\circ_q$ becomes singular there). 
This is a genus-zero curve with one cusp at $q=0$, one order-4 orbifold point at $q=\infty$, and a nodal point at $q^4=1/4^4$ (i.e., a point where $X^\circ_q$ develops a node).
\end{example}

\begin{example}
In contrast, deforming $q$ does give the full complex moduli space of the mirror quintic: we have $\mcpx(X^\circ) \cong \left(\CP^1 \setminus \{q=0,q^5=1/5^5\}\right)/(\Z/5)$, which is a genus-zero curve with one cusp at $q=0$, one order-$5$ orbifold point at $q=\infty$, and one nodal point at $q^5=1/5^5$ (also called the \emph{conifold point}).
\end{example}

The cusps in these examples are examples of \emph{large complex structure limits}: this means that they can be compactified by adding a singular variety which is `maximally degenerate' (in the sense that it is normal-crossings with a zero-dimensional stratum). 
In the first case, $\mathcal{M}_{1,1}$ can be compactified to $\overline{\mathcal{M}}_{1,1}$ by adding a nodal elliptic curve; in the second and third $\mcpx(X^\circ)$ can be partially compactified by adding the point $q=0$ corresponding to the singular variety $X^\circ_0 = \{\prod z_i = 0\}$. 

The K\"ahler moduli space is harder to define mathematically. 
More precisely, it is difficult to define the $A$-model invariants away from a certain part of the K\"ahler moduli space called the \emph{large volume limit} which is mirror to the large complex structure limit. 
It corresponds to complexified K\"ahler forms with $\omega \to +\infty$. 
We will see a precise definition of a formal neighbourhood of the large volume limit in \S \ref{sec:lvl}, together with an explanation of why it is hard to make mathematical sense of the $A$-model away from it.

\begin{example}
In examples \ref{eg:quartic} and \ref{eg:quintic}, the mirror map identifies the one-dimensional space of complex structures on $X^\circ$ parametrized by $q$ (in a formal neighbourhood of the large complex structure limit $q=0$)  with the one-dimensional space of complexified K\"ahler forms proportional to the Fubini--Study form (in a formal neighbourhood of the large volume limit). 
\end{example}

\begin{example}
\label{eg:mirr_quart_map}
It is more compelling to consider the `other direction' of mirror symmetry for the quartic: the $A$-model on the mirror quartic versus the $B$-model on the quartic. 
We have $22$ natural divisors on the mirror quartic, which we can naturally label by certain quartic monomials in four variables. 
Firstly, we have the four proper transforms of the images of the coordinate hyperplanes $\{z_i = 0\}$: we label these by the monomials $x_i^4$. 
Secondly, we have the eighteen exceptional divisors that were created when we resolved the $6$ $A_3$ singularities. 
The $A_3$ singularities lie on the pairwise intersections $\{z_i=z_j = 0\}$ of the coordinate hyperplanes, and it is natural to label the three exceptional divisors of the resolution by $x_i^3x_j$, $x_i^2x_j^2$, $x_i x_j^3$. 
Of course, divisors are Poincar\'e dual to classes in $H^{1,1}$, so we can deform the K\"ahler form on the mirror quartic in the direction of these $22$ divisors. 
We remark that there is some redundancy: the divisors only span a $19$-dimensional subspace of $H^{1,1}$. 
This corresponds, under the mirror map, to deforming the complex structure on the quartic by deforming the coefficients in front of the corresponding $22$ monomials of its defining equation. 
Once again there is some redundancy: up to isomorphism there is a $19$-dimensional space of quartics. 
\end{example}

Example \ref{eg:mirr_quart_map} is an example of the `monomial--divisor correspondence' \cite{Aspinwall1993b} which underlies the mirror map for Batyrev mirrors:

\begin{example}
\label{eg:Bat2}
Let's consider Batyrev's construction (Example \ref{eg:Batyrev}). 
The complex structure on the hyperplane section $X^\circ$ depends on the defining equation, which is a section of $\mathcal{L}^\circ$. 
The global sections of $\mathcal{L}^\circ$ have a basis indexed by lattice points of $\Delta^\circ$. 
The section corresponding to the origin cuts out $X^\circ_\infty$, the toric boundary of $V$, which is maximally degenerate so a large complex structure limit point.  
We can define nearby points in the complex moduli space by adding combinations of the sections corresponding to the remaining lattice points to the defining equation of the hypersurface. 
The remaining lattice points all lie on the boundary because the polytope is reflexive, thus we have parameters on $\mcpx(X^\circ)$ indexed by the boundary lattice points of $\Delta^\circ$.

On the other side of mirror symmetry, recall that the toric variety $V$ can be described either using the polytope $\Delta$, or using a fan $\Sigma$. 
This $\Sigma$ is the normal fan to the polytope, and it can be identified with the cone over the boundary of $\Delta^\circ$. 
The toric variety $V$ (and hence its hyperplane section $X$) will typically be singular, so we resolve it by subdividing the fan $\Sigma$. 
We do this by introducing rays spanned by boundary lattice points of $\Delta^\circ$: so the resolution has one divisor for each boundary lattice point. 
Of course, divisors define $(1,1)$ classes in cohomology, so we have parameters on $\mkah(X)$ indexed by the boundary lattice points of $\Delta^\circ$. 

To first order, the mirror map $\psi: \mkah(X) \to \mcpx(X^\circ)$ equates these two sets of parameters (more details can be found in \cite{Aspinwall1993b} or \cite[Section 6]{coxkatz}).
\end{example}

\subsection{Closed-string mirror symmetry: Hodge theory}

`Closed-string mirror symmetry' has to do with numerical invariants of Calabi--Yau K\"ahler manifolds. 
In the closed-string world, the $B$-model concerns periods (integrals of holomorphic forms over cycles), while the $A$-model concerns Gromov--Witten invariants. 
We recall that the Gromov--Witten invariants of $X$ count holomorphic maps of Riemann surfaces $u: \Sigma \to X$, possibly with some marked points and constraints, with the curve $u$ weighted by $\exp\left(2\pi \iii \cdot \omega^\C(u)\right)$. 
See Figure \ref{fig:GW}.\footnote{At first glance Gromov--Witten invariants look more like complex than symplectic invariants, since they are weighted counts of \emph{holomorphic} maps and $\omega^\C$ only enters via the weights; but it turns out they are invariant under deformations of the complex structure. 
They can even be defined using an \emph{almost-complex structure}, i.e., a not-necessarily-integrable complex structure $J: TX \to TX$, so long as $\omega(\cdot, J \cdot)$ is positive-definite.} 

\begin{figure}
\begin{center}
\hfill\includegraphics[scale=0.5]{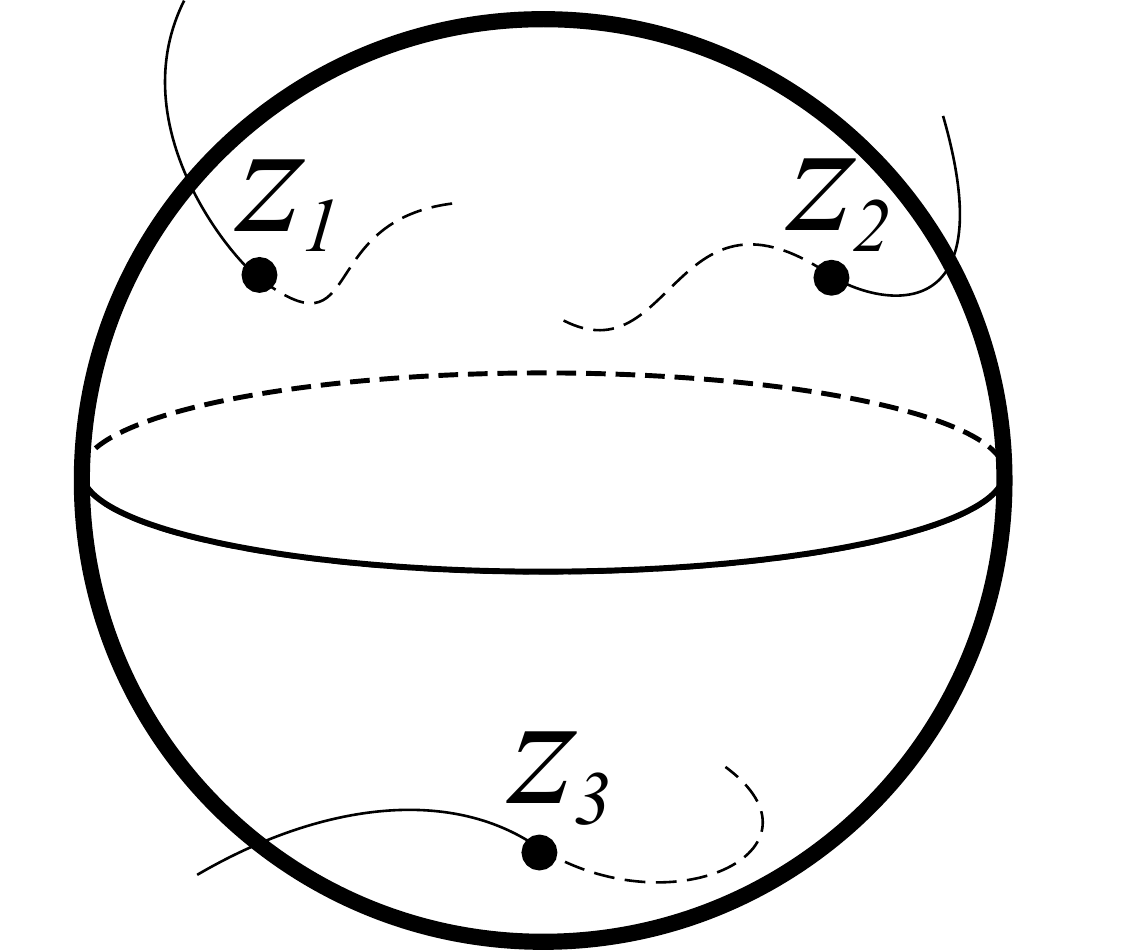}\hfill
\includegraphics[scale=0.5]{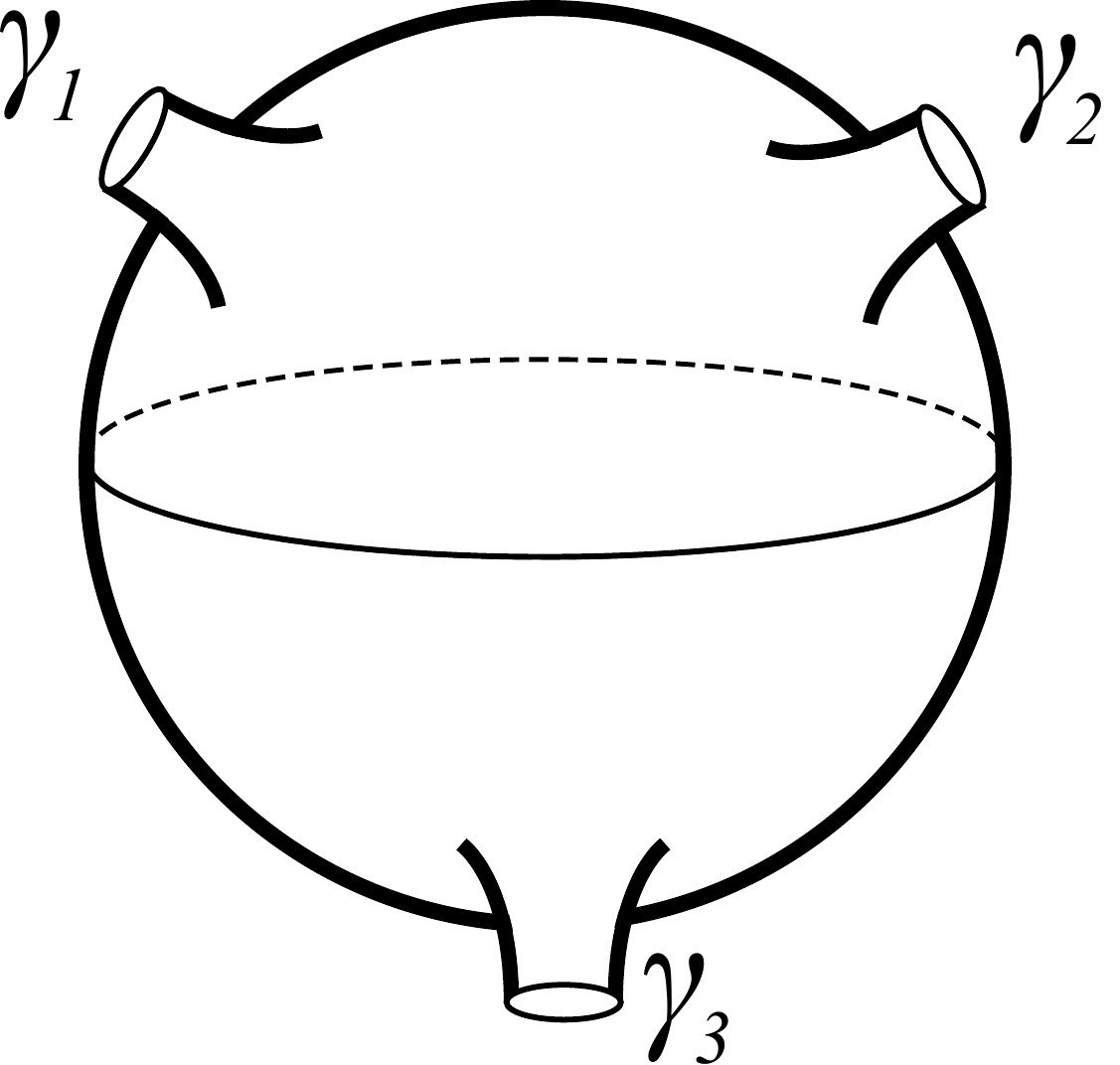}\hfill
\end{center}
\caption{\label{fig:GW}
Three-point Gromov--Witten invariants define a map $GW_{0,3}:H^*(X;\C)^{\otimes 3} \to \C$. Given inputs $\alpha_1,\alpha_2,\alpha_3 \in H^*(X;\Z)$, one chooses Poincar\'e dual cycles $A_1,A_2,A_3$ and distinct points $z_1,z_2,z_3 \in \CP^1$. 
When the degrees $|\alpha_i|$ add up to the dimension of $X$, the space of holomorphic maps $u: \CP^1 \to X$ such that $u(z_i) \in A_i$ for all $i$ is (after `perturbing to achieve transversality' \cite[\S 6.7]{mcduffsalamon}) a discrete but potentially infinite set.
We assign a weight $\exp(2\pi \iii \cdot \omega^\C(u))$ to each point $u$ in this set, and define $GW_{0,3}(\alpha_1,\alpha_2,\alpha_3)$ to be the sum of these weights. 
If the degrees don't add up to the dimension of $X$, such holomorphic maps may come in a higher-dimensional family; in this case we define $GW_{0,3}$ to be $0$. 
This is called a \emph{closed-string} invariant because it can equivalently be interpreted as a sum of maps $u$ as on the right, where the marked points are replaced by punctures, along which the map is asymptotic to `closed strings' $\gamma_i$ (as opposed to `open strings' which have boundary on `branes', see Figure \ref{fig:fuk}). 
Compare \S \ref{subsec:HF}.}
\end{figure}

Thus, if $(X,X^\circ)$ is a mirror pair then closed-string mirror symmetry relates the genus-zero Gromov--Witten invariants of $X$ to the periods of $X^\circ$. 
Periods are relatively easy to compute, whereas Gromov--Witten are relatively hard to compute. 
For this reason one often says that closed-string mirror symmetry yields predictions for the genus-zero Gromov--Witten invariants of $X$ in terms of the periods of $X^\circ$.

\begin{example}
Candelas, de la Ossa, Green and Parkes computed the mirror predictions for the genus-zero Gromov--Witten invariants of the quintic threefold $X$ \cite{Candelas1991}. 
Their predictions were verified by Givental \cite{Givental1996} and Lian--Liu--Yau \cite{Lian1997}.  
\end{example}

\begin{example}
\label{eg:cy2boring}
Closed-string mirror symmetry also works for Examples \ref{eg:ell_curve} and \ref{eg:quartic}, but is less interesting: it predicts that the genus-zero Gromov--Witten invariants of elliptic curves and quartic surfaces should be equal to zero, and indeed this is true. 
\end{example}

Although the most concrete formulation of closed-string mirror symmetry relates numerical invariants on the two sides of the mirror, the explicit relationship is rather complicated. 
Morrison showed how it can be formulated more conceptually by assembling these numerical invariants into algebraic structures: namely, \emph{variations of Hodge structures} \cite{Morrison1993} (see also \cite{Voisin1996,coxkatz}). 
For us, a variation of Hodge structures over $\mathcal{M}$, written $V \to \mathcal{M}$, consists of a vector bundle $\mathcal{E} \to \mathcal{M}$, equipped with a filtration $F^{\ge \bullet} \mathcal{E}$ and a flat connection $\nabla$ satisfying Griffiths transversality: $\nabla F^{\ge \bullet} \subset F^{\ge \bullet -1}$. 
Our Hodge structures also come with a \emph{polarization}, which is a covariantly constant pairing that respects the filtration in a certain way, but we'll omit it for simplicity.

The closed-string $B$-model is the classical variation of Hodge structures $V^B(X^\circ) \to \mcpx(X^\circ)$. 
The vector bundle has fibre $\mathcal{E}_p := H^*(X^\circ_p;\C)$, the filtration $F^{\ge \bullet}$ is the Hodge filtration, and the connection is the Gauss--Manin connection. 

The closed-string $A$-model is the so-called `$A$-variation of Hodge structures' $V^A(X) \to \mkah(X)$. 
The vector bundle has fibre $\mathcal{E}_p := H^*(X;\C)$ (it is trivial), the filtration is  $F^{\ge \bullet} \mathcal{E} := H^{-* \ge \bullet}(X;\C)$, and the connection is the \emph{Dubrovin} or \emph{Givental} connection: the connection matrix 
\[ \Gamma: H^{1,1}(X;\C) \to \End(H^*(X;\C))\]
can be identified with the restriction of the three-point Gromov--Witten invariant $GW_{0,3}$ (see caption to Figure \ref{fig:GW}) to $H^{1,1}(X;\C) \otimes H^*(X;\C)^{\otimes 2} \subset H^*(X;\C)^{\otimes 3}$, via the identification of $H^*(X;\C)$ with its dual arising from Poincar\'e duality.

Thus, closed-string mirror symmetry says that there is an isomorphism of variations of Hodge structures $V^A(X) \xrightarrow{\sim} V^B(X^\circ)$, covering the mirror map $ \mkah(X) \xrightarrow{\sim} \mcpx(X^\circ)$. 
We will only consider this isomorphism near the large volume/large complex structure limit points, due to the previously-mentioned difficulty with defining the Gromov--Witten invariants away from the large volume limit point. 
The monodromies of the connections around these points have a certain property called \emph{maximal unipotence}. 

There is a procedure for extracting numerical invariants from a variation of Hodge structures over a neighbourhood of a maximally unipotent limit point, which outputs Gromov--Witten invariants when applied to $V^A$ and periods when applied to $V^B$. 
In particular, closed-string mirror symmetry in the form of an isomorphism of variations of Hodge structures $V^A(X) \cong V^B(X^\circ)$ implies the relationship between Gromov--Witten invariants of $X$ and periods of $X^\circ$ referenced above.

\begin{rmk}
We have focused on closed-string mirror symmetry `at genus zero' (on the $A$-side, this means we only considered Gromov--Witten invariants counting maps $u: \Sigma \to X$ where $\Sigma$ is a curve of genus zero). 
There is a generalization to `higher-genus closed-string mirror symmetry' which relates higher-genus Gromov--Witten invariants to `BCOV theory' \cite{Bershadsky1994}. 
This has been used to predict higher-genus Gromov--Witten invariants, for example up to genus 51 on the quintic \cite{Huang2009}, and these predictions have been verified in genus one \cite{Zinger2009} and two \cite{Guo2017}.
\end{rmk}

\subsection{Open-string (a.k.a. homological) mirror symmetry: categories}

Homological mirror symmetry (HMS) was proposed by Kontsevich \cite{Kontsevich1994}. 
In the open-string world, rather than variations of Hodge structures, the relevant invariants are $A_\infty$ categories (we define these in \S \ref{subsec:ainfcat}). 
The $B$-model is denoted $\Dcoh(X^\circ)$: it is `a differential graded enhancement of the bounded derived category of coherent sheaves on $X^\circ$'. 
The $A$-model is denoted $\Dfuk\left(X,\omega^\C\right)$: it is `the split-closed bounded derived category of the Fukaya category of $\left(X,\omega^\C\right)$' (we will summarize the definition in \S \ref{subsec:fuk}). 
Thus, HMS predicts a (quasi-)equivalence of $A_\infty$ categories
\[ \Dfuk\left(X,\omega^\C_p\right) \simeq \Dcoh\left(X^\circ_{\psi(p)}\right)\]
where $\psi: \mkah(X) \xrightarrow{\sim} \mcpx(X^\circ)$ is the mirror map. 

\begin{rmk}
\label{rmk:hmsapps}
We have seen that closed-string mirror symmetry predicts interesting enumerative invariants which have been proved in many cases, but are there similar reasons to have `faith' in HMS? 
To start with it has been established in a (relatively small) number of examples: \cite{Abouzaid2017} builds on \cite{Polishchuk1998,Kontsevich2001,Fukaya2002a,Abouzaid2010d} to prove HMS for nonsingular SYZ torus fibrations; \cite{Seidel2003,Sheridan2015} prove it for hypersurfaces in projective space; \cite{Sheridan2017a} proves it for `generalized Greene--Plesser mirrors'; and generalizations to many non-compact and non-Calabi--Yau examples have also been proved. 
In cases where it has been established, it has been used to prove results about the structure of symplectic mapping class groups \cite{Sheridan2017b} and classification of Lagrangian submanifolds \cite{Abouzaid2010d}, by translating them into tractable questions in the world of algebraic geometry. 

Perhaps more importantly, it has turned out to be extremely powerful as a motivating principle (compare Remark \ref{rmk:wiigf}). 
Motivation from HMS has led to fundamental advances in algebraic geometry, such as the construction of braid group actions on derived categories of coherent sheaves \cite{Thomas2001}, stability conditions on triangulated categories \cite{Bridgeland2008,Bridgeland2006b}, and knot invariants \cite{Cautis2008}; and in symplectic geometry, such as results on symplectic mapping class groups \cite{Seidel2008d}, and the development of quilts \cite{Mau2016}; furthermore, it has been observed that it is intimately related with the geometric Langlands program \cite{Hausel2003}. 
This is far from an exhaustive list. 
\end{rmk}

Besides its power as a motivating principle, Kontsevich emphasized the explanatory power of HMS \cite{Kontsevich1994}:

\begin{quote}
``Our conjecture, if it is true, will unveil the mystery of Mirror Symmetry. The numerical predictions mean that two elements of an uncountable set (formal power series with integral coefficients) coincide. Our homological conjecture is equivalent to the coincidence in a countable set (connected components of the `moduli space of triangulated categories', whatever it means).''
\end{quote}

In other words, it is `less of a coincidence' that HMS should be true, than that closed-string mirror symmetry should be true; since HMS implies closed-string mirror symmetry (at least at genus zero, as we will see in \S \ref{subsec:hmsimplies}), this goes some way towards `explaining' closed-string mirror symmetry.
One of our main aims in this survey is to expand on this idea, and explain how it leads to an efficient approach to proving (both versions of) mirror symmetry.

\subsection{Using versality}
\label{subsec:vers}

Let $X^\circ$ be a Calabi--Yau K\"ahler manifold. 
The Bogomolov--Tian--Todorov theorem says that $X^\circ$ admits a local universal deformation $\EuX^\circ \to \mcpx$ with $\mcpx$ smooth, $T_0 \mcpx \cong H^1(X^\circ,TX^\circ)$.  
This means that any local family of deformations of $X^\circ$ parametrized by a base $B$ is pulled back via a classifying map $B \to \mcpx$, which is unique. 

More straightforwardly, if $X$ is Calabi--Yau K\"ahler manifold with complexified K\"ahler class $\left[\omega^\C\right]$, then it is obvious that we have a universal family of complexified K\"ahler classes parametrized by a neighbourhood of $\left[\omega^\C\right] \in H^{1,1}(X;\C)$.

\begin{strat}
\label{strat:vers}
Suppose that we had proved a version of mirror symmetry (closed- or open-string) at one point in moduli space:
\[ A(X,\omega^\C_p) \cong B(X^\circ_q) \quad \text{ for some $p \in \mkah(X)$, $q \in \mcpx(X^\circ)$},\]
and furthermore that the corresponding deformations of the invariants (Hodge structures in the closed-string case, categories in the open-string case) parametrized by the moduli spaces were universal in a neighbourhood of the points $p$ and $q$.
Then it would follow by universality that there existed a mirror map $\psi$ identifying a neighbourhood of $p$ with a neighbourhood of $q = \psi(p)$, and an isomorphism 
\[ A(X,\omega^\C_{p'}) \cong B(X^\circ_{\psi(p')}).\]
So mirror symmetry would be proved for all $p'$ in a neighbourhood of $p$.
\end{strat}

A fundamental difference between closed-string and homological mirror symmetry is that, whereas the variations of Hodge structure we consider are usually not universal, the categories are. 
So Strategy \ref{strat:vers} can work to prove homological mirror symmetry, but not to prove closed-string mirror symmetry. 
 
Indeed, variations of Hodge structures $V \to \cM$ can be classified by a period map $\cM \to \mathcal{P}$ where $\mathcal{P}$ is a `period domain' \cite{Griffiths1968}. 
Griffiths transversality says that the differential of the period map lies in the \emph{horizontal distribution}. 
If the differential mapped $T\cM$ isomorphically to the horizontal distribution, we would have a universal variation of Hodge structures: the image of the period map would be the unique maximal integral submanifold for the horizontal distribution. 
This happens for Calabi--Yaus of dimension one and two, if we take the polarizations on our Hodge structures into account in the latter case, but it cannot happen for a Calabi--Yau of dimension $\ge 3$: the dimension of the horizontal distribution is always strictly greater than that of $\mcpx$ (see \cite[Remark 10.20]{Voisin2002a}).
 
To clarify the situation, let us choose local coordinates in $\mathcal{P}$ so that the image of the period map can be written as the graph of a function. 
The horizontal distribution determines a differential equation satisfied by this function, which \emph{under-constrains} it because the dimension of the submanifold is less than that of the distribution. 
Thus, if we were trying to prove closed-string mirror symmetry by identifying the images of the period maps classifying $V^A$ and $V^B$, and had identified the corresponding functions to order one million, they might still differ at higher order. 
This is why Kontsevich says that closed-string mirror symmetry means that ``two elements of an uncountable set coincide''. 

In fact the procedure for extracting numerical invariants from a variation of Hodge structures consists, roughly, in choosing these local coordinates in a standard way and then taking the Taylor coefficients of the corresponding function. 
That is why closed-string mirror symmetry doesn't give interesting numerical information for Calabi--Yaus of dimension $\le 2$ (Example \ref{eg:cy2boring}): there's `no room' for interesting Taylor coefficients in this situation.

The situation for HMS, on the other hand, is different: as we will explain in \S \ref{sec:versint}, in many situations the family of derived categories $\Dcoh(X^\circ_p)$ parametrized by $p \in \mcpx(X^\circ)$, or of Fukaya categories $\Dfuk(X,\omega^\C_p)$ parametrized by $p \in \mkah(X)$, is universal (more precisely, versal). 
Thus, Strategy \ref{strat:vers} \emph{can} work to prove HMS. 
Furthermore, we will explain in \S \ref{subsec:hmsimplies} that HMS implies closed-string mirror symmetry in a certain sense: so from the versal invariant (the category), we can extract the non-versal invariant (the variation of Hodge structures), from which the enumerative invariants can be extracted as the Taylor coefficients `measuring its failure to be versal'.

\subsection{Versality at the large volume limit}

There are two problems with Strategy \ref{strat:vers} for proving HMS, as stated. 
First, as we have already mentioned, it is difficult to make sense of the $A$-model (in this case, the Fukaya category) away from the large volume limit. 
We will see in \S \ref{sec:lvl} that in fact, we can only define the $A$-model in a \emph{formal neighbourhood} of the large volume limit -- so there are no interior points $p \in \mkah(X)$ where the program could get started. 
Second, even if we could make sense of the Fukaya category at an interior point of the moduli space, it's not clear that establishing HMS at that particular point would be any easier than establishing it at every point.

The solution to both of these is to employ Strategy \ref{strat:vers} \emph{at the large volume limit point}. 
In other words we construct compactifications $\mkah \subset \mbarkah$ including the large volume limit point $p$, and $\mcpx \subset \mbarcpx$ including the large complex structure limit point $q$. 
We then start by proving HMS at these limit points, then try to extend to the formal neighbourhood by versality. 
We will see, however, that the versal deformation space is typically \emph{not} smooth at these boundary points, which is an obstruction to applying versality. 
However, in \S \ref{sec:versbound} we explain an ad-hoc way to fix this by modifying the deformation problem, which we expect to work in many cases (e.g., many Calabi--Yau hypersurfaces in toric varieties). 

Thus, in certain circumstances, we reduce the problem of proving HMS to the problem of proving it in the large volume/large complex structure limit. 
We will explain what this means more explicitly in \S \ref{sec:provehms}. 

\begin{rmk}
\label{rmk:enumless}
Let us explain why this makes HMS `easier to prove'.  
One reason it's hard to prove mirror symmetry is that the $A$-model is hard to compute: it involves enumerating holomorphic maps from Riemann surfaces into $X$, and there's no systematic way to do so. 
The original proofs of closed-string mirror symmetry for complete intersections in toric varieties employed equivariant localization (with respect to the algebraic torus action on the ambient toric variety) to do this enumeration. 
One of the nice things about the versality strategy we will outline is that it reduces the number of enumerations we need to a finite number: rather than computing infinitely many Taylor coefficients separately, we only need the zeroth- and first-order Taylor coefficients to apply a versality theorem.
\end{rmk}

We conclude the paper by outlining, in \S \ref{sec:skel}, an approach to proving HMS in the large volume/large complex structure limit which was suggested by Kontsevich and has recently been developed by Nadler. 
The idea is to reduce computation of the Fukaya category in the large volume limit to a sheaf-theoretic computation, which can be phrased in terms of microlocal sheaf theory on the `Lagrangian skeleton'. 
Unfortunately it's not immediately clear how to input this to the versality results that we describe, but we speculate on what might be needed to make this connection. 

\begin{rmk}
This paper is longer on ideas than precise statements of theorems, by design. 
Our main aim is to paint an overall picture, which would be obscured by extra notation and hypotheses if we tried to make precise statements, so we have mostly omitted them and instead given references to the literature where they can be found. 
It's worth emphasizing that many of the basic objects that we discuss have only been constructed under certain hypotheses (and/or are work in preparation). 
We do not enumerate these hypotheses, but rather talk as if everything has been constructed in full generality; the reader who wants to prove theorems about these objects will have to consult the references to see the degree of generality in which they have actually been constructed. 
\end{rmk}

\paragraph{Acknowledgments:} I am supported by a Royal Society University Research Fellowship. I am very grateful to my collaborators Strom Borman, Sheel Ganatra, Tim Perutz and Ivan Smith, from whom I've learned a lot about this material, and in particular to Sheel Ganatra who helped me with \S \ref{sec:skel} (mistakes are my own of course). I'm also very grateful to Helge Ruddat, who pointed out the connection with versal deformation spaces of $d$-semistable $K3$ surfaces \cite{Friedman1983}; and to Siu-Cheong Lau, who suggested that anti-symplectic involutions could be used to constrain deformations of Fukaya categories (which is what \S \ref{subsec:cut} is about); and to John Lesieutre, for a conversation about ample cones. Finally, my work on this subject wouldn't have gotten started if Paul Seidel hadn't explained to me how to define the coefficient ring $R_{X,D}$ of the relative Fukaya category, and many other things.


\section{Deformation theory via $L_\infty$ algebras}
\label{sec:def}

We recall that many natural deformation problems can be phrased in terms of differential graded Lie algebras \cite{Manetti2005a}, or more generally $L_\infty$ algebras \cite{Kontsevich2003,Getzler2009,Kontsevichh}. 
This framework is fundamental for our study of deformation spaces of the categories involved in HMS, so in this section we summarize the important background. 
We work over a field $\Bbbk$ of characteristic zero throughout.

\subsection{$L_\infty$ algebras}
\label{subsec:linf}

Recall that an $L_\infty$ algebra $\fg$ is a graded vector space equipped with multilinear maps
\begin{align*}
\ell^s:\fg^{\otimes s} & \to \fg \quad \text{for all $s \ge 1$, denoted}\\
v_1 \otimes \ldots \otimes v_s & \mapsto \{v_1,\ldots,v_s\}.
\end{align*}
These maps should have degree $2-s$, they should be graded symmetric with respect to reduced degrees (which means that swapping two adjacent inputs $v$ and $w$ changes the sign by $(-1)^{(|v|+1)(|w|+1)}$), and they should satisfy the $L_\infty$ relations:
\[ \sum_{j,\sigma} (-1)^\dagger \{\{v_{\sigma(1)},\ldots,v_{\sigma(j)}\},v_{\sigma(j+1)},\ldots,v_{\sigma(s)}\},\]
where the sum is over all $j$ and all `unshuffles' $\sigma$: i.e., permutations $\sigma$ satisfying $\sigma(1)<\ldots<\sigma(j)$ and $\sigma(j+1)<\ldots<\sigma(s)$. 
The sign $\dagger$ is the Koszul sign associated to commuting the inputs $v_i$ through each other, equipped with their reduced degrees $|v_i|+1$ as before. 
If $\ell^{\ge 3} = 0$, then $\fg$ is called a \emph{differential graded Lie algebra}. 

The first of the $L_\infty$ relations says that $\ell^1(\ell^1(v)) = 0$, so $\ell^1$ is a differential: we denote its cohomology by $H^*(\fg)$. 
The second says that $\ell^2$ satisfies the Leibniz rule, so defines a bracket on $H^*(\fg)$. 
The third says that $\ell^2$ satisfies the Jacobi relation up to a homotopy given by $\ell^3$, and in particular the bracket on $H^*(\fg)$ satisfies the Jacobi relation.

There is a notion of $L_\infty$ morphism $f: \fg \dashrightarrow \fh$, which consists of multilinear maps $f^s: \fg^{\otimes s} \to \fh$ for $s \ge 1$ satisfying certain axioms \cite{Kontsevich2003}. 
The first of these says that $f^1: (\fg,\ell^1) \to (\fh,\ell^1)$ is a chain map: if it induces an isomorphism on cohomology, then $f$ is called an $L_\infty$ quasi-isomorphism and we say that $\fg$ and $\fh$ are quasi-isomorphic. 
This is an equivalence relation: $L_\infty$ quasi-isomorphisms can be composed and inverted \cite[Theorem 4.6]{Kontsevich2003}.

An $L_\infty$ algebra $\fh$ is called \emph{minimal} if $\ell^1 = 0$.  
Any $L_\infty$ algebra admits a \emph{minimal model}, i.e., a quasi-isomorphic $L_\infty$ algebra which is minimal \cite[Lemma 4.9]{Kontsevich2003}. 

\subsection{Maurer--Cartan elements}

Given an $L_\infty$ algebra $\fg$, we would like to define the set of solutions to the Maurer--Cartan equation:
\[ MC_\fg := \left\{\alpha \in \fg^1: \sum_{s\ge 1} \frac{1}{s!}\{\underbrace{\alpha,\ldots,\alpha}_s\} = 0\right\}, \]
but the infinite sum may not converge, so we need to be a bit careful.
What we \emph{can} define is $MC_\fg(R)$, where $(R,\fm)$ is a complete Noetherian local $\Bbbk$-algebra. 
It is the set of $\alpha \in \fg^1 \hat{\otimes} \fm$ satisfying the Maurer--Cartan equation, where the hat on the tensor product means we complete with respect to the $\fm$-adic filtration after tensoring. 
The Maurer--Cartan equation then converges by completeness. 
 
For any $\gamma(t)$ in the $\fm$-adic completion of $\fg^0 \otimes \fm[t]$, there is a corresponding $t$-dependent vector field $v(\gamma,t)$ on $\fg^1 \hat{\otimes} \fm$, defined by
\[v(\gamma,t)_\alpha := \sum_{i\ge 0} \frac{1}{i!}\{\gamma(t),\underbrace{\alpha,\ldots,\alpha}_i\}.\]
One can verify that this vector field is tangent to $MC_\fg(R)$, and that its flow exists for all times. 
We say that two solutions of the Maurer--Cartan equation are \emph{gauge equivalent} if they are connected by a finite series of flowlines of such vector fields: we denote this equivalence relation by $\sim$. 
The following theorem is fundamental in the study of deformation problems via $L_\infty$ algebras (see \cite[Theorem 2.4]{Goldman1988}, \cite[Theorem 4.6]{Kontsevich2003}, \cite[Theorem 2.2.2]{Fukaya2003}):

\begin{thm}
\label{thm:mcqi}
An $L_\infty$ quasi-isomorphism $f: \fg \dashrightarrow \fh$ induces a bijection 
\[ MC_\fg(R)/\!\! \sim \,\,\,\longrightarrow\, MC_\fh(R)/\!\! \sim.\]
\end{thm}

\subsection{Versal deformation space}

\begin{defn}
We say that $\alpha \in MC_\fg(S)$ is
\begin{itemize}
\item \emph{Complete} if any $\beta \in MC_\fg(R)$ is gauge-equivalent to $\psi^* \alpha$ for some $\psi^*:S \to R$;
\item \emph{Universal} if furthermore this $\psi^*$ is uniquely determined;
\item \emph{Versal} if instead $\psi^*: \fm_S/\fm^2_S \to \fm_R/\fm_R^2$ is uniquely determined.
\end{itemize}
\end{defn}

We start by studying versal deformation spaces of a \emph{minimal} $L_\infty$ algebra $\fg$, which we will assume to be finite-dimensional in each degree for simplicity (this will be the case in all of our applications in this paper).  
Choose a basis $\vec{e}_1,\ldots,\vec{e}_a$ for $\fg^1$ and $\vec{f}_1,\ldots,\vec{f}_b$ for $\fg^2$.  
Define $P_j \in \Bbbk\power{x_1,\ldots,x_a}$ by
\begin{align} 
\label{eqn:Pj} \sum_{s \ge 1} \frac{1}{s!}\{\underbrace{\alpha_v,\ldots,\alpha_v}_s \}& = \sum_{j=1}^b P_j(x_1,\ldots,x_a) \cdot \vec{f}_j, \quad \text{ where} \\
\label{eqn:alphavers} \alpha_v &:= \sum_{i=1}^a x_i \cdot \vec{e}_i.
\end{align}

\begin{defn}
We define the complete Noetherian local $\Bbbk$-algebra
\begin{align*}
R_v & := \Bbbk\power{x_1,\ldots,x_a}/(P_1,\ldots,P_b).
\end{align*}
\end{defn}

There is a tautological Maurer--Cartan element $\alpha_v \in MC_\fg(R_v)$, defined by \eqref{eqn:alphavers}, and any element of $MC_\fg(R)$ is equal to $\psi^* \alpha_v$ for a unique $\psi^*:R_v \to R$. 
It is immediate that $\alpha_v$ is complete, and it is also versal; however it need not be universal, because there may be a gauge equivalence $ \alpha_v \sim \psi^* \alpha_v$ for some $\psi^*:R_v \to R_v$ which induces the identity map on $\fm/\fm^2$ but differs from the identity at higher order. 
Thus it makes sense to define the \emph{versal deformation space} of $\fg$ to be 
\[ MC_\fg := \spec(R_v).\]

If $\fg$ is an arbitrary $L_\infty$ algebra with finite-dimensional cohomology in each degree, then we can construct a minimal model $\fh$ for $\fg$. 
We can then construct a versal Maurer--Cartan element for $\fh$ via the above construction, which yields one for $\fg$ by Theorem \ref{thm:mcqi}.

\begin{rmk}
We say that $\fg$ is \emph{homotopy abelian} if it is quasi-isomorphic to an $L_\infty$ algebra $\fh$ which has all operations $\ell^s$ equal to $0$. 
In this case the versal deformation space is very simple: the $P_j$ in \eqref{eqn:Pj} vanish, so the versal deformation space is a formal neighbourhood of $0 \in H^1(\fg)$.
\end{rmk}

\subsection{Versality criterion}

Any Maurer--Cartan element $\alpha \in MC_\fg(R)$ determines a \emph{Kodaira--Spencer map}
\[ KS_\alpha: \left(\fm/\fm^2\right)^* \to H^1(\fg).\]
Explicitly, the projection of $\alpha$ to $\fg^1 \otimes \fm/\fm^2$ is $\ell^1$-closed by the Maurer--Cartan equation, so defines an element in $H^1(\fg) \otimes \fm/\fm^2$. 
The Kodaira--Spencer map is defined to be contraction with this element. 

\begin{thm}
\label{thm:vers}
Let $R = \Bbbk\power{\nov_1,\ldots,\nov_k}$ be a formal power series algebra, and $\beta \in MC_\fg(R)$. 
If $KS_\beta$ is an isomorphism (respectively, surjective) then $\beta$ is versal (respectively, complete).
\end{thm}
\begin{proof}
Let $\fh$ be a minimal model for $\fg$. 
Then $\beta$ corresponds to some $\alpha \in MC_\fh(R)$ under Theorem \ref{thm:mcqi}. 
Therefore $\alpha = \psi^* \alpha_v$ for some $\psi^*: R_v \to R$. 
The composition
\[\Bbbk\power{x_1,\ldots,x_a} \to \Bbbk\power{x_1,\ldots,x_a}/(P_1,\ldots,P_b) = R_v \xrightarrow{\psi^*} R = \Bbbk\power{\nov_1,\ldots,\nov_k}\]  
defines a map $\fm_v/\fm_v^2 \to \fm/\fm^2$ which is an isomorphism (respectively, injective) by hypothesis, because $KS_\alpha$ can be identified with $KS_\beta$. 
It follows by the inverse function theorem that the composition is an isomorphism (respectively, left-invertible). 
Therefore $\alpha_v = \phi^*\alpha$ for some $\phi^*: R \to R_v$, so $\alpha$ is complete. 
If $KS_\alpha$ is an isomorphism then $\phi^*$ defines an isomorphism $\fm/\fm^2 \xrightarrow{\sim} \fm_v/\fm_v^2$, from which it follows that $\alpha$ is furthermore versal. 
\end{proof}

\section{Deformations of $A_\infty$ categories}
\label{sec:defainfcat}

We introduce basic definitions of $A_\infty$ categories, following \cite[Chapter I]{Seidel2008} and \cite{fooo}, and their deformation theory (compare \cite{Seidel2003,Sheridan2017}). 

\subsection{$A_\infty$ categories}
\label{subsec:ainfcat}

If $R$ is a ring, we will define an $A_\infty$ algebra $\cA$ over $R$ to be a free graded $R$-module equipped with multilinear maps
\begin{align*}
\mu^s: \cA^{\otimes s} & \to \cA 
\end{align*}
for $s \ge 1$, of degree $2-s$, satisfying the $A_\infty$ relations:
\begin{equation}
\label{eqn:ainf}
 \sum (-1)^{|a_1|+\ldots+|a_i|+i} \cdot \mu^{s+1-j}(a_1,\ldots,a_i,\mu^{j}(a_{i+1},\ldots,a_{i+j}),a_{i+j+1},\ldots,a_s) = 0.
\end{equation}

The first $A_\infty$ relation says that $\mu^1(\mu^1(a)) = 0$, so $\mu^1$ is a differential; we denote its cohomology by $H^*(\cA)$. 
The second says that $\mu^2$ satisfies the Leibniz rule, so defines a product on $H^*(\cA)$. 
The third says that $\mu^2$ is associative up to a homotopy given by $\mu^3$, and in particular the product on $H^*(\cA)$ is associative, so $H^*(\cA)$ is an $R$-algebra.

There is a notion of $A_\infty$ morphism $F: \cA \dashrightarrow \cB$, which consists of multilinear maps $F^s: \cA^{\otimes s} \to \cB$ for $s \ge 1$ satisfying 
\begin{equation}
\label{eqn:ainfmorph}
 \sum_{i,j} F^{s+1-j}(a_1,\ldots,a_i,\mu_\cA^j(\ldots,a_{i+j}),\ldots,a_s) = \sum_{i_1,\ldots,i_k}\mu_\cB^{k}(F^{i_1}(a_1,\ldots,a_{i_1}),\ldots,F^{i_k}(\ldots,a_s)).
\end{equation}
There is an induced homomorphism of $R$-algebras $H^*(F): H^*(\cA) \to H^*(\cB)$.  
If it is an isomorphism we call $F$ a \emph{quasi-isomorphism}. 
As a special case, if $F^1$ is an isomorphism we call $F$ an isomorphism. 
Morphisms can be composed, and any $A_\infty$ isomorphism admits an inverse (i.e., a morphism $G: \cB \to \cA$ such that $G \circ F = \id$; in particular, not just an `inverse up to homotopy').
If $F^{\ge 2} = 0$, we say that $F$ is \emph{strict}.

There is an analogous notion of an $A_\infty$ category. 
An $R$-linear $A_\infty$ category $\cA$ consists of a set of objects, morphism spaces which are free graded $R$-modules $hom^*_\cA(K,L)$ for each pair of objects, and $A_\infty$ structure maps
\[ \mu^s: hom_\cA^*(L_0,L_1) \otimes \ldots \otimes hom_\cA^*(L_{s-1},L_s) \to hom_\cA^*(L_0,L_s) \]
satisfying \eqref{eqn:ainf}. 
One can define the cohomological category $H^*(\cA)$ by analogy with the case of algebras. 
The notion of $A_\infty$ morphism extends to a notion of $A_\infty$ functor; an $A_\infty$ functor between $A_\infty$ categories induces an honest functor between the cohomological categories. 
If this functor is an equivalence (respectively, an embedding, i.e. fully faithful), we call the $A_\infty$ functor a \emph{quasi-equivalence} (respectively, a \emph{quasi-embedding}). 
If the functor is bijective on objects and $F^1$ is an isomorphism on each morphism space, we call $F$ an isomorphism. 

\subsection{Curved $A_\infty$ categories}
\label{subsec:curvedainf}

There is a notion of \emph{curved} $A_\infty$ categories, which is exactly the same except one allows the existence of $\mu^0$ in the definition. 
The $A_\infty$ relations \eqref{eqn:ainf} no longer involve finitely many terms for each $s$, so one needs a reason for the infinite sum to converge. 
We will deal with this issue in the same way that we did for the Maurer--Cartan equation: namely by considering $A_\infty$ categories of the form $\cA_0 \hat{\otimes} R$ where $(R,\fm)$ is a complete Noetherian local $\Bbbk$-algebra, and requiring $\mu^0 \in \cA_0 \hat{\otimes} \fm$. 
The $A_\infty$ equations \eqref{eqn:ainf} then converge by completeness. 
There similarly exists a notion of a curved $A_\infty$ functor between curved $A_\infty$ categories, which is a set of maps $F^s$ for $s \ge 0$ satisfying \eqref{eqn:ainfmorph}, with $F^0 \in \cA_0 \hat{\otimes} \fm$. 

Note that our assumptions ensure that $\cA/\fm := \cA \otimes_R R/\fm$ is an (uncurved) $A_\infty$ category, and a curved functor $\cA \dashrightarrow \cB$ induces an uncurved functor $\cA/\fm \dashrightarrow \cB/\fm$.

Curved $A_\infty$ categories are not good objects to deal with. 
For example a curved $A_\infty$ category $\cA$ only has a well-defined cohomology category $H^*(\cA)$ if the curvature $\mu^0$ vanishes. 
In particular the notion of `quasi-equivalence' does not make sense for curved categories; since HMS is expressed as a quasi-equivalence between two $A_\infty$ categories, we really want to work in the uncurved world. 

The standard way to turn a curved $A_\infty$ category into an uncurved one is to form the category of \emph{objects equipped with bounding cochains} \cite{fooo}.
A bounding cochain $\alpha$ for an object $L$ of a curved $A_\infty$ category $\cA$ is an element $\alpha \in hom^1_{\cA_0}(L,L) \hat{\otimes} \fm$ satisfying the \emph{Maurer--Cartan equation}:\footnote{This is a different Maurer--Cartan equation from the one we considered for $L_\infty$ algebras.}
\[\sum_{s \ge 0} \mu^s(\alpha,\ldots,\alpha) = 0.\] 
If $L$ admits a bounding cochain, we say it is \emph{unobstructed}; otherwise we say it is \emph{obstructed}. 

There is an (uncurved) $A_\infty$ category $\cA^\bc$ whose objects consist of pairs $(L,\alpha)$ where $\alpha$ is a bounding cochain for $L$. 
If $\cA$ and $\cB$ are curved $A_\infty$ categories of the above type, then a curved $A_\infty$ functor $F: \cA \dashrightarrow \cB$ induces an (uncurved) $A_\infty$ functor 
\[F^\bc: \cA^\bc \dashrightarrow \cB^\bc.\]
If $F/\fm$ is a quasi-embedding, then so is $F^\bc$ by a comparison argument for the spectral sequences induced by the $\fm$-adic filtration.
We refer to \cite[\S 2.7]{Sheridan2017} for more details.

\subsection{Curved deformations of $A_\infty$ categories}

One might ask why we bothered to introduce curved $A_\infty$ categories and functors, since we already said they are not the objects we wish to deal with. 
The answer is that they fit naturally into the deformation theory framework outlined in \S \ref{sec:def}, whereas uncurved $A_\infty$ categories do not. 

Let $\cA_0$ be an $A_\infty$ category over $\Bbbk$, and $(R,\fm)$ a complete Noetherian local $\Bbbk$-algebra with $R/\fm = \Bbbk$. 
A \emph{deformation of $\cA_0$ over $R$} is a curved $A_\infty$ structure $\mu^*$ on $\cA := \cA_0 \hat{\otimes} R$, such that $\cA/\fm = \cA_0$. 
We say that deformations $\cA$ and $\cA'$ are \emph{equivalent} if there is an $A_\infty$ functor $F:\cA \dashrightarrow \cA'$ which is equal to the identity modulo $\fm$.

We introduce the space of \emph{Hochschild cochains on $\cA_0$}:
\begin{equation}
\label{eqn:hh}
 CC^*(\cA_0) := \prod_{L_0,\ldots,L_s} \Hom(hom_{\cA_0}^*(L_0,L_1) \otimes \ldots \otimes hom_{\cA_0}^*(L_{s-1},L_s), hom_{\cA_0}^*(L_0,L_s)).
\end{equation}
Equipping it with the Hochschild differential and the Gerstenhaber bracket turns it into a differential graded Lie algebra $\fg$. 
A deformation of $\cA_0$ over $R$ is equivalent to a Maurer--Cartan element $\alpha \in MC_\fg(\fm)$ for $\fg = CC^*(\cA_0)$. 
If two such Maurer--Cartan elements are gauge equivalent, then the corresponding deformations are equivalent.

\begin{rmk}
An uncurved deformation of $\cA_0$ over $R$ is equivalent to a Maurer--Cartan element with vanishing length-zero component: i.e., the part of the Maurer--Cartan element in the $s=0$ component $\prod_{L_0} hom_{\cA_0}^*(L_0,L_0)$ of \eqref{eqn:hh} is equal to $0$. 
These do not fit into the framework of \S \ref{sec:def} except in certain circumstances, for example when one knows \emph{a priori} (for example, for grading reasons) that $\alpha^0 = \alpha^1 = 0$, in which case one can instead use the \emph{truncated} Hochschild cochain complex \cite[\S 3b]{Seidel2003}.
\end{rmk}

\begin{rmk}
For this reason, the mantra `Hochschild cochains control deformations of an $A_\infty$ category' requires some caution. 
For example, even if $\cA_0$ is uncurved, there might exist curved deformations $\cA$ for which $\cA^\bc$ is empty (i.e., all objects might be obstructed). See \cite{Blanc2017} for a more detailed analysis.
\end{rmk}

\begin{thm}
\label{thm:ainfvers}
Let $\cB$ be a curved $A_\infty$ category over $R = \Bbbk\power{x_1,\ldots,x_a}$, such that the corresponding Kodaira--Spencer map
\[ KS_{\cB} : (\fm/\fm^2)^* \to HH^2(\cB_0)\]
is surjective. 
If $\cA$ is a curved $A_\infty$ category over a complete Noetherian local $\Bbbk$-algebra $(S,\fm)$ with $S/\fm = \Bbbk$, such that there is an $A_\infty$ isomorphism $\cB_0 \dashrightarrow \cA_0$, then there exists $\Psi^*: R \to S$ and an $A_\infty$ quasi-embedding
\[ (\Psi^* \cB)^\bc \hookrightarrow \cA^\bc.\] 
\end{thm}
\begin{proof}
One first shows that there is a curved $A_\infty$ functor $F: \Psi^* \cB \dashrightarrow \cA$ which reduces to the isomorphism $\cB_0 \dashrightarrow \cA_0$ modulo $\fm$. 
When $\cA_0 = \cB_0$, this is an immediate consequence of Theorem \ref{thm:vers} and the preceding discussion; to put ourselves in that situation, we use the fact that we can modify the $A_\infty$ structure on $\cA$ to make $F_0$ strict by \cite[\S 1c]{Seidel2008}. 

The resulting $A_\infty$ functor $F^\bc$ is then a quasi-embedding because $F/\fm$ is, by the discussion in \S \ref{subsec:curvedainf}.
\end{proof}

\begin{rmk}
\label{rmk:ainfiso}
Analogously to $L_\infty$ algebras, $A_\infty$ quasi-equivalences defined over a field can be inverted up to homotopy \cite[Corollary 1.14]{Seidel2008}.
Over a ring the same is not true, which is why we unfortunately need to assume that $\cA_0$ and $\cB_0$ are $A_\infty$ isomorphic (rather than quasi-equivalent) in our theorem. 
It's likely possible to prove a generalization that works for quasi-equivalences, but I don't know precisely how to formulate it. 
\end{rmk}

\section{The large volume limit}
\label{sec:lvl}

In \S \ref{subsec:moduli} we saw examples of the compactification $\mbarcpx$ of the complex moduli space, which includes the large complex structure limit point corresponding to a maximally degenerate variety. 
In this section we will discuss the corresponding compactification $\mbarkah$ of the K\"ahler moduli space, which includes the `large volume limit point' where $\omega = +\infty$. 
We can give a more precise definition of a \emph{formal neighbourhood} of the large volume limit point $\mbarhatkah(X)$, and its intersection $\mhatkah(X)$ with $\mkah(X)$. 
We will explain why the $A$-model can typically only be defined over this formal neighbourhood.

\begin{example}
If the K\"ahler moduli space is one-dimensional then $\mhatkah \subset \mbarhatkah$ will be isomorphic to a formal  punctured disc $\spec (\C \laurents{q})$ sitting inside the formal disc $\spec (\C\power{q})$.
\end{example}

\subsection{Closed-string: Gromov--Witten invariants}

\subsubsection{One complexified K\"ahler form}

We will not go into the precise definition of the Gromov--Witten invariants, which is rather technical. 
However, the simple fact that they are defined by a sum of counts of holomorphic maps $u: \Sigma \to X$, weighted by $\exp\left(2\pi\iii \cdot \omega^\C\left(u\right)\right)$ (where $\omega^\C(u) := \int_\Sigma u^* \omega^\C$), allows us to see how one should define the large volume limit of the K\"ahler moduli space.

To start with, we observe that the magnitude of this weight is $\exp\left(-2\pi \cdot \omega(u)\right)$, so one might hope that the sums converge if $\omega^\C$ is close enough to the large volume limit because the weights become very small as $\omega \to +\infty$.
However, actually proving convergence would require estimates on the growth rate of Gromov--Witten invariants which are difficult to obtain, so it doesn't make sense to bake them into the theory. 
Instead we work in a \emph{formal} neighbourhood of the large volume limit. 

The idea is to start with a complexified K\"ahler form $\omega^\C = B + \iii\omega$, and consider the family of complexified K\"ahler forms $\omega^\C_q := B + (\log(q)/2\pi \iii) \cdot \omega$ in the limit $q \to 0^+$. 
A curve $u: \Sigma \to X$ then gets weighted by 
\[ e^{2\pi\iii \cdot \omega^\C_q(u)} = q^{\omega(u)} \cdot e^{2\pi \iii \cdot B(u)}.\]
We regard these weights as elements of the \emph{Novikov field} $\Lambda := \C \laurents{q^\R}$ (the algebra of complex combinations of real powers of $q$, completed with respect to the $q$-adic filtration). 
Convergence of our sums in the Novikov field requires that for every $N$, there are finitely many curves $u$ with $\omega(u) \le N$. 
This is true by Gromov's compactness theorem for holomorphic curves \cite{Gromov1985}.

\begin{rmk}
\label{rmk:onlykahlclass}
Observe that, given the complex manifold $X$, the Gromov--Witten invariants only depend on the cohomology class $\left[\omega^\C\right]$ of the complexified K\"ahler form $\omega^\C$. 
In fact, one can easily see that they only depend on the image of $\left[\omega^\C\right]$ in $H^2(X;\C)/H^2(X;\Z)$, since adding an integral class to the $B$-field does not change the weight $e^{2\pi \iii \cdot B(u)}$.
\end{rmk}

\subsubsection{Many complexified K\"ahler forms}
\label{subsubsec:all}

Following these ideas through, we can give a precise definition of a formal neighbourhood of the large volume limit point of the K\"ahler moduli space.
If $X$ is a K\"ahler manifold, we set 
\[ \overline{NE}(X) := \{ u \in H_2(X): \omega(u) \ge 0 \text{ for all K\"ahler forms $\omega$}\}\]
where $H_2(X) := H_2(X;\Z)/\text{torsion}$ ($\overline{NE}(X)$ is the `closure of the cone of effective curve classes on $X$'). 
This cone is strongly convex, so the group ring $\C\left[\overline{NE}(X)\right]$ has a unique maximal ideal $\fm$. 
We denote the completion at this ideal by $R_X := \C\power{\overline{NE}(X)}$.

\begin{defn}
The formal neighbourhood of the large volume limit point of the K\"ahler moduli space of $X$ is
\[ \mbarhatkah(X) := \spec \left( R_X \right).\] 
\end{defn}

Let's explain why this is a good definition. We can define the Gromov--Witten invariants to count holomorphic curves $u: \Sigma \to X$ with a weight $\nov^{[u]} \in \C\left[\overline{NE}(X)\right]$. 
Gromov compactness ensures that these infinite weighted sums converge in the completed ring, so the Gromov--Witten invariants define regular functions on $\mbarhatkah(X)$. 
Furthermore, a complexified K\"ahler form $\omega^\C$ defines a $\Lambda$-point $p$ of this scheme:
\begin{align}
\label{eqn:lambdapoint} p^*: R_X & \to \Lambda \\
\nonumber p^*\left(\nov^u \right) & := q^{\omega(u)} \cdot  e^{2\pi \iii \cdot B(u)},
\end{align}
and the Gromov--Witten invariants of $\left(X,\omega^\C\right)$ are obtained by specializing to this $\Lambda$-point. 

\subsubsection{At the large volume limit point}
\label{subsubsec:gwlvl}

The large volume limit point is the $\C$-point $0 \in \mbarhatkah(X)$ corresponding to the maximal ideal $\fm$. 
The specialization of the Gromov--Witten invariants to the large volume limit point is rather simple: a curve $u$ gets weighted by $1$ if it is constant, and by $0$ if it is non-constant. 
So the Gromov--Witten invariants at the large volume limit point count constant maps of Riemann surfaces to $X$: they essentially only contain information about intersection theory on $X$.

\begin{rmk}
Let's assume for the moment that $\overline{NE}(X)$ is rational polyhedral, so $\C[\overline{NE}(X)]$ is the coordinate ring of the affine toric variety corresponding to its dual cone (the nef cone). 
The formal neighbourhood of the large volume limit point is, by definition, the formal neighbourhood of the unique torus-fixed point in this toric variety corresponding to the vertex of the nef cone. 
The toric divisors correspond to degenerate K\"ahler forms, and their complement (the intersection of the formal neighbourhood with the dense torus orbit) to honest K\"ahler forms.
The dense torus orbit has coordinate ring $\C[H_2(X)] \supset \C[\overline{NE}(X)]$. 
Thus it makes sense to define
\[ \mhatkah(X) := \spec \left( R_X \otimes_{\C\left[\overline{NE}(X)\right]} \C[H_2(X)]\right).\]
Note that the $\Lambda$-point corresponding to a complexified K\"ahler form $\omega^\C$ lies in $\mhatkah(X) \subset \mbarhatkah(X)$, but the large volume limit point does not.
\end{rmk}


\subsection{Open-string: Fukaya category}
\label{subsec:fuk}

We summarize the definition of the Fukaya category (see \cite{Auroux2013} for an excellent introduction). 
We will start by assuming that the $B$-field vanishes, so $\omega^\C = \iii\omega$. 
We will explain how to incorporate a non-zero $B$-field into the definition in \S \ref{subsubsec:B}, but remark that it will be ignored for the rest of the paper so the uninterested reader can safely skip that section.

\subsubsection{One K\"ahler form}
\label{subsubsec:onekah}

If $\omega^\C = \iii \omega$, then the objects of the Fukaya category are \emph{Lagrangian branes}, which consist of:
\begin{itemize}
\item a Lagrangian submanifold $L \subset X$: i.e., a half-dimensional submanifold with $\omega|_L = 0$.
\item a spin structure on $L$, and a grading of $L$ (in the sense of \cite{Seidel1999}).
\end{itemize}
We will not talk further about the spin structure (which allows us to define \emph{signed} counts of holomorphic curves) or grading (which allows us to define \emph{gradings} on our morphism spaces). 

The morphism space between transversely-intersecting Lagrangian branes $L_0,L_1$ is the (graded) vector space with basis elements indexed by intersection points:
\[ hom^*_{\fuk\left(X,\omega^\C\right)}(L_0,L_1) := \bigoplus_{p \in L_0 \cap L_1} \Bbbk \cdot p\]
for some coefficient field $\Bbbk$ to be specified.\footnote{
Of course not every pair of Lagrangians intersects transversely, but there exist technical workarounds that we won't go into (for example, one can make an auxiliary choice of perturbations of each pair of Lagrangians making them transverse \cite{Seidel2008}).}

The ``composition maps'' in this category are $\Bbbk$-linear graded maps
\[ \mu^s: hom^*(L_0,L_1) \otimes hom^*(L_1,L_2) \otimes \ldots \otimes hom^*(L_{s-1},L_s) \to hom^*(L_0,L_s)[2-s].\]
They are defined by giving their matrix coefficients in terms of the bases of intersection points: the coefficient of $p_0 \in L_0 \cap L_s$ in $\mu^s(p_1,\ldots,p_s)$ is the (signed) count of holomorphic maps $u: \mathbb{D} \to X$, where $\mathbb{D} \subset \C$ is the closed unit disc with boundary points $\zeta_0,\zeta_1,\ldots \zeta_s$ removed; we require that $u(\partial_i \mathbb{D}) \subset L_i$, and $u(z) \to p_i$ as $z\to \zeta_i$ (see Figure \ref{fig:fuk}). 

\begin{figure}
\begin{center}
\hfill\includegraphics[scale=0.5]{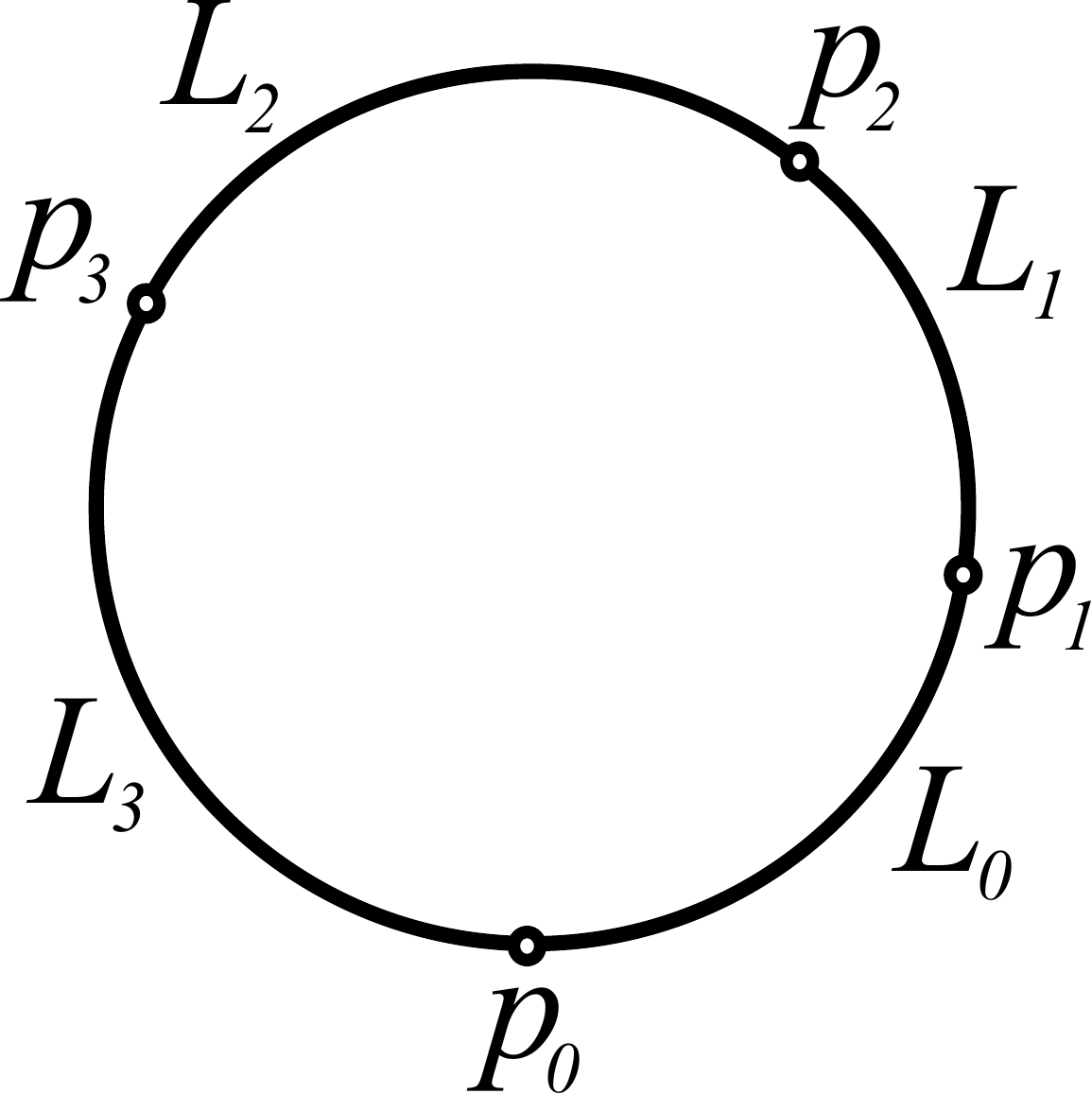}
\hfill\includegraphics[scale=0.5]{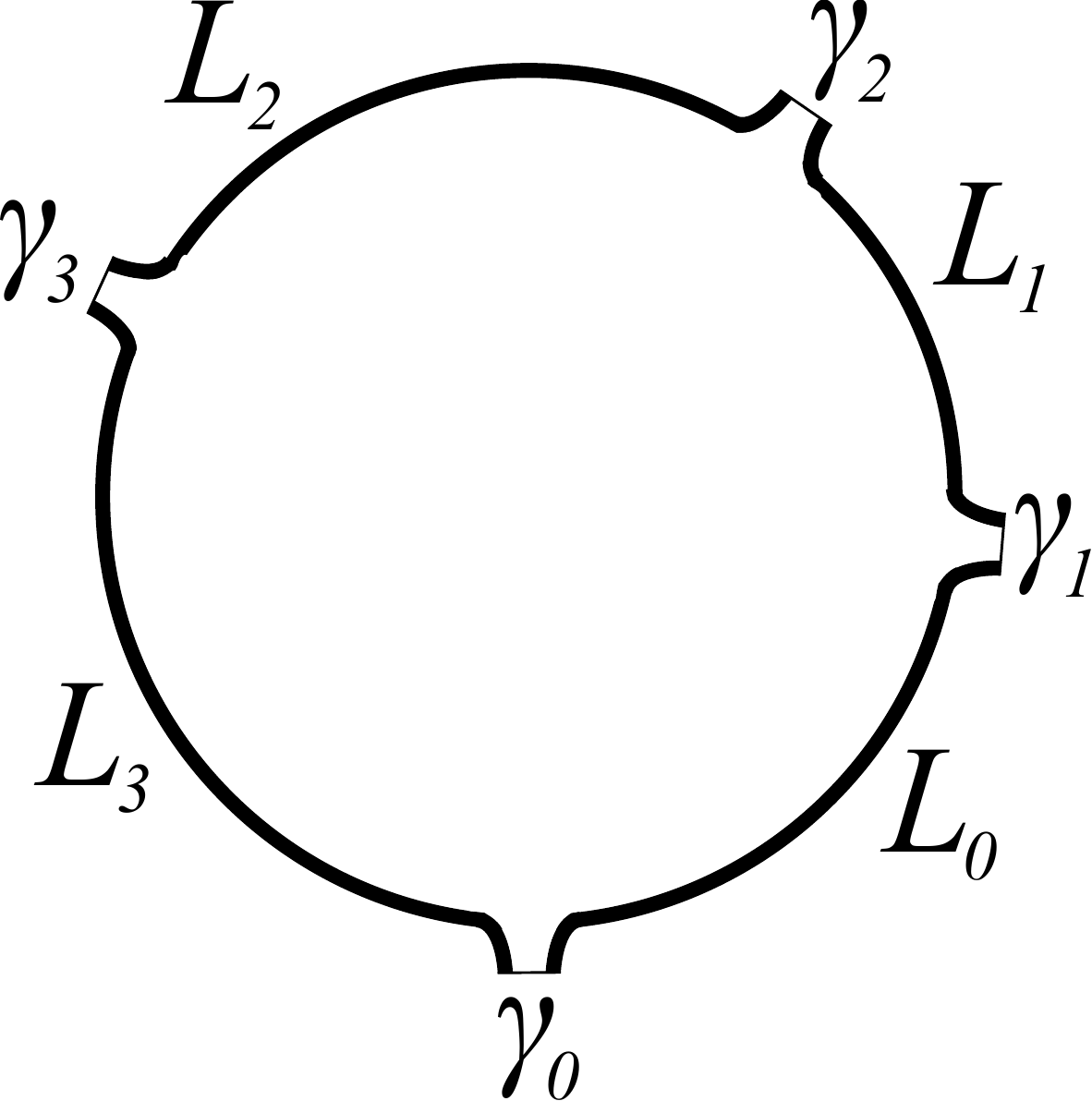}\hfill
\caption{\label{fig:fuk} On the left, a picture of a holomorphic disc contributing to the coefficient of $p_0 \in hom^*(L_0,L_3)$ in $\mu^3(p_1,p_2,p_3)$. On the right, an explanation of why these are called `open-string' invariants: the intersection points $p_i$ can equivalently be interpreted as constant `strings with boundary' $\gamma_i$ stretching between the `branes' $L_i$. 
Compare Figure \ref{fig:GW}.}
\end{center}
\end{figure}

Similarly to the Gromov--Witten invariants, we would like to count the disc $u$ with weight $\exp(2\pi \iii \cdot \omega^\C(u))$ (where again $\omega^\C(u) := \int_{\mathbb{D}} u^*\omega^\C$). 
However, as happened for the Gromov--Witten invariants, we have no good reason for the weighted sum to be finite, even close to the large volume limit. 
Therefore we are forced to work over the Novikov field $\Lambda$, and weight curves by 
\[ e^{2\pi \iii \cdot \omega^\C_q(u)} = q^{\omega(u)} \in \Lambda.\]
Thus we must take $\Bbbk = \Lambda$ above. 

The composition maps $\mu^s$ satisfy the $A_\infty$ relations:
\[ \sum_{i,j} \pm \mu^{s+1-j}(p_1,\ldots,p_i,\mu^j(p_{i+1},\ldots,p_{i+j}),p_{i+j+1},\ldots,p_s)=0,\]
so the Fukaya category $\fuk\left(X,\omega^\C\right)$ is a \emph{$\Lambda$-linear $A_\infty$ category}. 
This category may be curved, and we would like to turn it into an uncurved $A_\infty$ category using bounding cochains as in \S \ref{subsec:curvedainf}. 
To put ourselves in an appropriate context to introduce bounding cochains, we use the $q$-adic filtration on morphism spaces. 
For this we need to observe that the Fukaya category can in fact be defined over the complete Noetherian local $\C$-algebra $\Lambda_0 := \C \power{q^{\R_{\ge 0}}}$: that is because the holomorphic discs $u$ that define its composition maps have non-negative symplectic area $\omega(u) \ge 0$.\footnote{We're sweeping some thorny technical details under the rug here: for example, the perturbations that one needs to introduce to deal with non-transverse intersections of Lagrangians can interfere with this positivity of symplectic areas (the failure of positivity can however be made arbitrarily small by choosing sufficiently small perturbations). Moreover the invariance properties of the Fukaya category over the Novikov ring are much subtler than over the Novikov field.}
We then consider bounding cochains in
\[ hom^1_{\fuk\left(X,\omega^\C\right)}(L,L)_{>0} := \bigoplus_{p \in L_0 \cap L_1} \Lambda_{>0} \cdot p \subset hom^1_{\fuk\left(X,\omega^\C\right)}(L,L),\]
where $\Lambda_{>0} \subset \Lambda$ is the maximal ideal consisting of Laurent series with strictly positive powers of $q$. 
This ensures the convergence of the Maurer--Cartan equation, by completeness of the $q$-adic filtration, so we can define the corresponding uncurved $\Lambda_0$-linear $A_\infty$ category of bounding cochains, then tensor with $\Lambda$ to obtain the uncurved $\Lambda$-linear $A_\infty$ category $\fuk\left(X,\omega^\C\right)^\bc$.

\begin{rmk}
\label{rmk:fukonlykah}
Given the complex manifold $X$, the Fukaya category $\fuk\left(X,\omega^\C\right)^\bc$ only depends on the K\"ahler class $[\omega] \in H^2(X;\R)$ up to quasi-equivalence (compare Remark \ref{rmk:onlykahlclass}).
That is because the space of K\"ahler forms with the same K\"ahler class is convex, so any two are symplectomorphic by Moser's theorem, and the Fukaya category is a symplectic invariant.
\end{rmk}

Finally, let us recall that the category appearing on one side of HMS was denoted `$\Dfuk$' rather than `$\fuk$'. 
If $\cA$ is an $A_\infty$ category (which we'll assume to be linear over a field), then $\mathsf{DA}$ denotes the smallest $A_\infty$ category containing $\cA$ which is \emph{triangulated} and \emph{split-closed} (it is denoted $\Pi(Tw \cA)$ in \cite[\S 4c]{Seidel2008}). 
Since $\Dcoh(X^\circ)$ has these properties, HMS will only have a chance to be true if we enlarge $\fuk$ to $\Dfuk$ in this way. 
One can think of this formal enlargement as an algebraic substitute for having to include all sorts of singular Lagrangians in one's definition of the Fukaya category, which might cause serious analytical difficulties. 
It is a fact that there is a quasi-isomorphism of differential graded Lie algebras $CC^*(\mathsf{D}\cA) \dashrightarrow CC^*(\cA)$, so curved deformations of $\cA$ are equivalent to curved deformations of $\mathsf{D}\cA$.

\subsubsection{Many K\"ahler forms}

It is not immediately clear how to carry out the analogue of \S \ref{subsubsec:all} for the Fukaya category. 
Whereas Gromov--Witten invariants count maps $u: \Sigma \to X$ defining a homology class $[u] \in H_2(X)$, to which we can assign a weight $\nov^{[u]} \in \C\left[\overline{NE}(X)\right]$, the discs $u$ that we count in the definition of the Fukaya category don't define a homology class $[u]$ in any fixed homology group. 

The solution to this problem was explained by Seidel \cite{Seidel2002}: we use the \emph{relative} Fukaya category $\fuk(X,D)$. 
This depends on an auxiliary piece of data: a normal-crossings divisor $D \subset X$ which supports an effective ample divisor. 
We will only consider certain Lagrangians that avoid $D$, so the discs $u$ that we count in the definition of the Fukaya category define homology classes $[u] \in H_2(X,X \setminus D)$.

\begin{example}
Let $X$ be a hypersurface in a toric variety $V$, as in Batyrev's mirror construction (Example \ref{eg:Batyrev}). 
We will use the divisor $D \subset X$ that is the intersection of $X$ with the toric boundary divisor of $V$ in this case. 
\end{example}

More precisely, we suppose that $\alpha \in \Omega^1(X \setminus D;\R)$ is a one-form such that $\omega|_{X \setminus D} = d\alpha$. 
This defines a lift of $[\omega] \in H^2(X;\R)$ to $[\omega;\alpha] \in H^2(X,X \setminus D;\R)$, by
\[ [\omega;\alpha](u) := \int_u \omega - \int_{\partial u} \alpha.\]
We observe that, if $D = \cup_{p \in P} D_p$ is the decomposition of $D$ into irreducible components, then $H^2(X,X \setminus D;\R)$ has a basis given by the Poincar\'e duals to the $D_p$. 
Therefore we have 
\[ [\omega;\alpha] = \sum_p \lambda_p \cdot PD(D_p)\]
for some $\lambda_p$. 
If $\lambda_p > 0$ for all $p$, then we call the pair $(\omega,\alpha)$ a \emph{relative K\"ahler form}.

\begin{rmk}
The condition $\lambda_p > 0$ for all $p$ is called \emph{convexity at infinity}. 
This is related to $X \setminus D$ being Stein: for example, if $\alpha = d^c \rho$ for some plurisubharmonic $\rho: X \setminus D \to \R$, then convexity at infinity means $\rho$ is exhausting (bounded below and proper). 
In fact, there exists a relative K\"ahler form with cohomology class $\sum_p \lambda_p \cdot PD(D_p)$ if and only if the $\R$-divisor $\sum_p \lambda_p \cdot D_p$ is effective and ample \cite[Lemma 3.3]{Sheridan2017}. 
\end{rmk}

We now define 
\[ \overline{NE}(X,D) := \{u \in H_2(X,X \setminus D): u \cdot E \ge 0 \text{ for all effective ample $\R$-divisors $E$ supported on $D$}\}.\]
We observe that any holomorphic curve $u: (\Sigma, \partial \Sigma) \to (X,X \setminus D)$ has homology class $[u] \in \overline{NE}(X,D)$.\footnote{If $u$ were a map from a smooth curve $\Sigma$ which is not contained in $D$, we would have $u \cdot E \ge 0 $ for all \emph{effective} divisors $E$ supported on $D$ by positivity of intersection. 
The reason we restrict to ample divisors has to do with the fact that we must take into account nodal curves $\Sigma$, some of whose components may get mapped inside $D$. 
Such components may have negative intersection number with components of $D$, but will still have positive intersection number with an ample class.}
As before, the algebra $\C[\overline{NE}(X,D)]$ has a unique toric maximal ideal $\fm$, and we complete at it to get $R_{X,D} := \C\power{\overline{NE}(X,D)}$. 

\begin{defn}
The formal neighbourhood of the large volume limit point of the relative K\"ahler moduli space is
\[ \mbarhatkah(X,D) := \spec\left(R_{X,D} \right).\]
\end{defn}

Observe that there is a map $\mbarhatkah(X,D) \to \mbarhatkah(X)$, induced by $H_2(X) \to H_2(X,X \setminus D)$: this is the forgetful map from the moduli space of complexified relative K\"ahler forms on $(X,D)$ to the moduli space of complexified K\"ahler forms on $X$.

The objects of the Fukaya category are \emph{relative Lagrangian branes}, which consist of:
\begin{itemize}
\item a Lagrangian brane $L \subset X$, avoiding $D$.
\item a function $h: L \to \R$ with $dh = \alpha|_L$.
\end{itemize}

The morphism spaces are 
\[ hom^*_{\fuk(X,D)}(L_0,L_1) := \bigoplus_{p \in L_0 \cap L_1} R_{X,D} \cdot p.\]
The $A_\infty$ structure maps count holomorphic discs $u$ as before, now weighted by $\nov^{[u]} \in R_{X,D}$. 

Analogously to what happened for the Gromov--Witten invariants, the relative K\"ahler class $[\omega;\alpha]$ determines a $\Lambda$-point $p$ of $\mbarhatkah(X,D)$ by
\begin{align*}
p^*: R_{X,D} & \to \Lambda \\
p^*\left(\nov^u\right) & := q^{[\omega;\alpha](u)}.
\end{align*}
Furthermore, there is an embedding
\[ \fuk(X,D)_p \hookrightarrow \fuk(X,\omega),\]
where the subscript $p$ means we take the fibre of the category over the $\Lambda$-point $p$. 
Explicitly, this means we tensor each morphism space with $\Lambda$, via the morphism $p^*$. 

The embedding is obvious on the level of objects: it simply forgets the function $h$.
It sends a morphism $p \in L_0 \cap L_1$ of $\fuk(X,D)_p$ to the morphism $q^{h_1(p) - h_0(p)} \cdot p$ of $\fuk(X,\omega)$. 
One can check that this embedding respects the $A_\infty$ maps: the structure maps count the same holomorphic discs, and a simple application of Stokes' theorem shows that each disc gets counted with the same weight in $\Lambda$ \cite[\S 8.1]{Sheridan2015}.

By a similar argument, there is an embedding
\[ \left(\fuk(X,D)^\bc\right)_p \hookrightarrow \fuk(X,\omega)^\bc.\]
One needs to be a little bit careful to ensure that the image of an order-$\fm$ bounding cochain is a positive-energy bounding cochain, because $q^{h_1(p)-h_0(p)}$ may be a negative power of $q$: in general the embedding only exists after one applies the reverse Liouville flow to the Lagrangians for a sufficiently long time (see \cite[Remark 5.22]{Sheridan2017}).
 
\subsubsection{At the large volume limit point}
\label{subsubsec:atlvl}

By analogy with what we saw for the Gromov--Witten invariants, the large volume limit point $0 \in \mbarhatkah(X,D)$ is defined to be the $\C$-point corresponding to the toric maximal ideal $\fm$. 
The fibre $\fuk(X,D)_0$ of the relative Fukaya category over the large volume limit point is the $\C$-linear $A_\infty$ category with the same objects, counting holomorphic discs with weight $1$ if they do not intersect $D$ and weight $0$ if they intersect $D$.
In other words, it is $\fuk(X \setminus D)$, the Fukaya category of the affine variety $X \setminus D$. 
This $A_\infty$ category is uncurved: this follows from the fact that the Lagrangians we consider are exact in $X \setminus D$, so only bound constant holomorphic discs.

It is then clear that $\fuk(X,D)$ is a (possibly curved) deformation of $\fuk(X\setminus D)$ over $R_{X,D}$. 

\subsubsection{Complexified K\"ahler forms}
\label{subsubsec:B}

Now we explain how to define $\fuk\left(X,\omega^\C\right)$ in the case that the $B$-field is non-vanishing: $\omega^\C = B + \iii \omega$ (compare \cite{Kapustin2004a}). 
Objects are Lagrangian branes $L$ as before, now equipped with a complex vector bundle $\mathcal{E}_L$ with a unitary connection $\nabla_L$, whose curvature is required to be $\id \otimes B|_L$. 
The morphism spaces are now 
\[ hom^*_{\fuk\left(X,\omega^\C\right)}(L_0,L_1) := \bigoplus_{p \in L_0 \cap L_1} \Hom(\mathcal{E}_{0,p},\mathcal{E}_{1,p}).\]
The $A_\infty$ composition maps count holomorphic discs $u$ as in \S \ref{subsubsec:onekah}. 
Each such disc determines a `monodromy' map
\[\mathrm{mon}(u): \Hom(\mathcal{E}_{0,p_1},\mathcal{E}_{1,p_1}) \otimes \ldots \otimes \Hom(\mathcal{E}_{s-1,p_s},\mathcal{E}_{s,p_s}) \to \Hom(\mathcal{E}_{0,p_0},\mathcal{E}_{s,p_0})\]
by composition interspersed with parallel transport maps
\[ \mathcal{E}_{i,p_i} \to \mathcal{E}_{i,p_{i+1}}\]
along $\partial_i \mathbb{D}$ with respect to $\nabla_i$.
The contribution of the disc $u$ to the composition map $\mu^s$ is then $q^{\omega(u)} \cdot e^{2\pi \iii \cdot B(u)} \cdot \mathrm{mon}(u)$.

\begin{rmk}
The situation considered in \S \ref{subsubsec:onekah} corresponded to the case $B=0$, and each $\mathcal{E}$ was the trivial line bundle with trivial connection.
\end{rmk}

\begin{rmk}
The condition that the curvature of each $\nabla_L$ is $\id \otimes B|_L$ ensures that $e^{2\pi \iii B(u)} \cdot \mathrm{mon}(u)$ is invariant under isotopies of $u$, by the relationship between parallel transport and curvature.
\end{rmk}

\begin{rmk}
Given the complex manifold $X$, the Fukaya category $\fuk\left(X,\omega^\C\right)$ only depends up to equivalence on the image of the complexified K\"ahler class $\left[\omega^\C\right]$ in $H^2(X;\C)/H^2(X;\Z)$ (compare Remark \ref{rmk:onlykahlclass}). 
In light of Remark \ref{rmk:fukonlykah}, to show this it suffices to check that for any closed real 2-form $B'$ representing an integral cohomology class, adding $B'$ to the $B$-field does not change the Fukaya category up to equivalence. 
Indeed, there exists a complex line bundle $\mathcal{V}$ with unitary connection $\nabla_{\mathcal{V}}$ on $X$ having curvature form $B'$: so there is an equivalence
\begin{align*}
\fuk(X,B+\iii\omega) & \to \fuk(X,B'+B+\iii\omega), \quad\text{ which on the level of objects sends} \\
(L,(\mathcal{E}_L,\nabla_L)) & \mapsto (L,(\mathcal{V},\nabla_{\mathcal{V}})\otimes (\mathcal{E}_L,\nabla_L) ).
\end{align*}
\end{rmk}

Now we consider the relative situation. 
A complexified relative K\"ahler form is a complexified K\"ahler form $\omega^\C$ together with $\theta^\C = \beta + \iii \alpha \in \Omega^1(X \setminus D;\C)$ satisfying $\omega^\C|_{X \setminus D} = d\theta^\C$, such that $(\omega,\alpha)$ is a relative K\"ahler form. 
This defines a lift of $\left[\omega^\C\right] \in H^2(X;\C)$ to $[\omega^\C;\theta^\C] \in H^2(X,X \setminus D;\C)$ as before, and hence a $\Lambda$-point $p$ of $\mbarhatkah(X,D)$ by
\begin{align*}
p^*: R_{X,D} & \to \Lambda \\
p^*\left(\nov^u\right) & := q^{[\omega;\alpha](u)} \cdot e^{2\pi\iii \cdot [B;\beta](u)}.
\end{align*}
There is an embedding
\begin{equation}
\label{eqn:embrelcomp}
 \fuk(X,D)_p \hookrightarrow \fuk\left(X,\omega^\C\right)
\end{equation}
as before: on the level of objects, it sends 
\[ L \mapsto (L,(\underline{\C},\nabla_\beta)|_L)\]
where $\nabla_\beta$ is the connection on the trivial line bundle $\underline{\C}$ on $X \setminus D$ with unitary connection $d + 2\pi \iii \beta$.

\begin{rmk}
Note that we did not need to alter the definition of $\fuk(X,D)$ for the embedding \eqref{eqn:embrelcomp} to exist: the relative Fukaya category `already knew' about the $B$-fields. 
We note that it would be natural to extend the definition of $\fuk(X,D)$ to include unitary vector bundles with flat connections on the Lagrangians.  
\end{rmk}

\section{Versality in the interior of the moduli space}
\label{sec:versint}

\subsection{Bogomolov--Tian--Todorov theorem}

Let $X$ be a complex projective manifold. 
Then there is a differential graded Lie algebra $\fg_{cs}$, with cohomology $H^*(\fg_{cs}) \cong H^*(T_X)$, which controls the deformations of $X$ \cite{Manetti2004}. 
In other words, for any complete Noetherian local $\C$-algebra $R$, deformations of $X$ over $R$ up to isomorphism are in bijection with Maurer--Cartan elements in $MC_{\fg_{cs}}(R)$ up to gauge equivalence. 
 
The Bogomolov\cite{Bogomolov1978}--Tian\cite{Tian1987}--Todorov\cite{Todorov1989} theorem can be phrased as saying that if $X$ is Calabi--Yau then $\fg_{cs}$ is homotopy abelian, so the versal deformation space of $X$ is a formal neighbourhood of $0 \in H^1(T_X)$ \cite{Goldman1990,Iacono2010}. 

There is a larger differential graded Lie algebra $\fg \supset \fg_{cs}$, with cohomology $H^*(\fg) \cong HT^*(X) := H^*(\wedge^* T_X)$. 
It controls (curved) deformations of the $A_\infty$ category $\Dcoh(X)$ in accordance with \S \ref{sec:defainfcat}, because there is a quasi-isomorphism
\[ I^*: \fg \dashrightarrow CC^*(\Dcoh(X)) \] 
(this is the composition of Kontsevich's formality quasi-isomorphism $\fg \dashrightarrow CC^*(X)$ \cite{Kontsevich2003} with the quasi-isomorphism $CC^*(X) \dashrightarrow CC^*(\Dcoh(X))$ constructed in \cite{Lowen2005}). 
It also is homotopy abelian when $X$ is Calabi--Yau (see \cite{Kontsevich2003,Barannikov1998} for the complex case), so the versal deformation space of $\Dcoh(X)$ is a formal neighbourhood of $0$ in the vector space
\[ HH^2(\Dcoh(X)) \cong HT^2(X).\]
Deformations of the category `coming from deformations of the complex structure' correspond to $H^1(T_X) \subset HT^2(X)$.

\begin{rmk}
Recall that we're only describing \emph{curved} deformations of the category: we are not claiming that all objects are unobstructed under all such deformations. 
Our approach later when actually studying deformations of $\Dcoh(X)$ will be to fix a deformation, fix a generating subcategory which happens to remain uncurved under that deformation, and focus attention on that subcategory. 
Thus we bypass the problem of dealing with obstructed objects, albeit in an essentially ad-hoc way. 
\end{rmk}

\subsection{Floer homology}
\label{subsec:HF}

Let $X$ be a Calabi--Yau K\"ahler manifold with complexified K\"ahler form $\omega^\C = B + \iii \omega$. 
Then there is an $L_\infty$ algebra $CF^*(X)$, called the \emph{Hamiltonian Floer cochains on $X$}, whose cohomology is $HF^*(X) \cong H^*(X;\Lambda)$. 

The cochain complex $(CF^*(X),\ell^1)$ was introduced by Floer \cite{Floer1989}, and should be thought of as a semi-infinite Morse complex for a certain `action functional' on the free loop space $\mathcal{L} X := \{ \gamma: S^1 \to X\}$. 
Explicitly, we choose a function $H: S^1 \times X \to \R$, and consider the corresponding (multi-valued) action functional $\mathcal{A}_H$ on $\mathcal{L}X$, whose differential is given by
\[ d\mathcal{A}_H(\xi) := \int_{S^1} \omega(\gamma'(t)-V_{H_t},\xi(t)) dt\]
where $\gamma \in \mathcal{L} X$, $\xi \in T_\gamma \mathcal{L} X$ is identified with a section $\xi$ of $\gamma^* TX$ by abuse of notation, and $V_{H_t}$ is the `Hamiltonian vector field corresponding to $H_t$', characterized by the property that $\omega(V_{H_t},-) = dH_t$. 
Critical points of this functional are loops $\gamma \in \mathcal{L} X$ which are time-1 orbits of the vector field $V_{H_t}$:
\begin{align*}
\gamma: S^1 & \to X \\
\gamma'(t) &= V_{H_t}.
\end{align*}
The Floer cochains are the free $\Lambda$-vector space spanned by such orbits:
\[ CF^*(X) := \bigoplus_\gamma \Lambda \cdot \gamma,\]
graded by Conley--Zehnder index (we refer to \cite{Salamon1999} for details). 

The $L_\infty$ operations $\ell^s$ on $CF^*(X)$ are constructed analogously to the $A_\infty$ operations in the Fukaya category: the coefficient of $\gamma_0$ in $\ell^s(\gamma_1,\ldots,\gamma_s)$ is the count of (pseudo-)holomorphic maps $u: \Sigma \to X$ with domain as in Figure \ref{fig:linf}, asymptotic to the orbits $\gamma_i$ as shown, weighted by $q^{\omega(u)}\cdot e^{2\pi \iii B(u)}$ (compare \cite{Fabert2013}). 

\begin{figure}
\begin{center}
\hfill\includegraphics[scale=0.5]{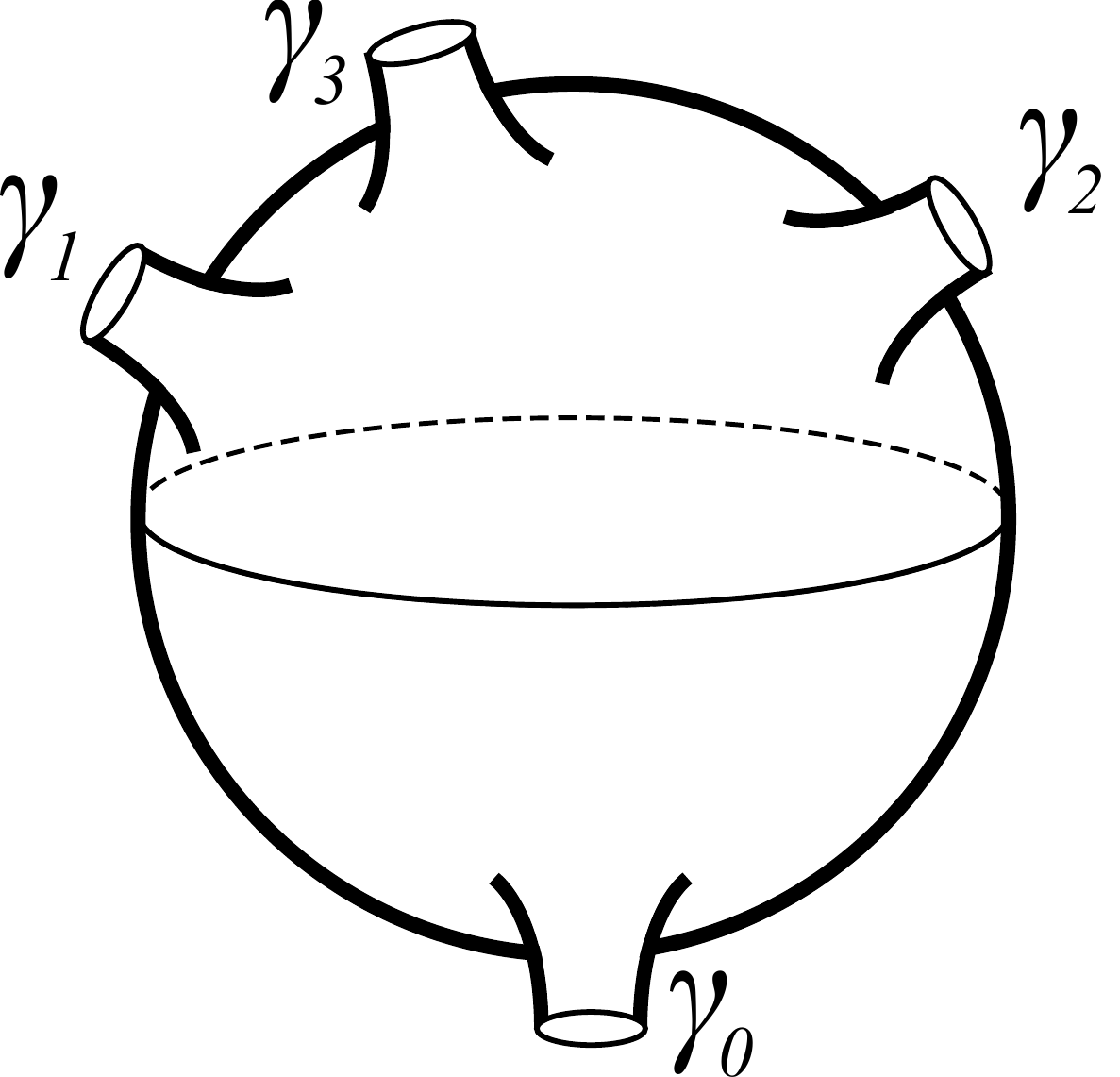}\hfill
\includegraphics[scale=0.5]{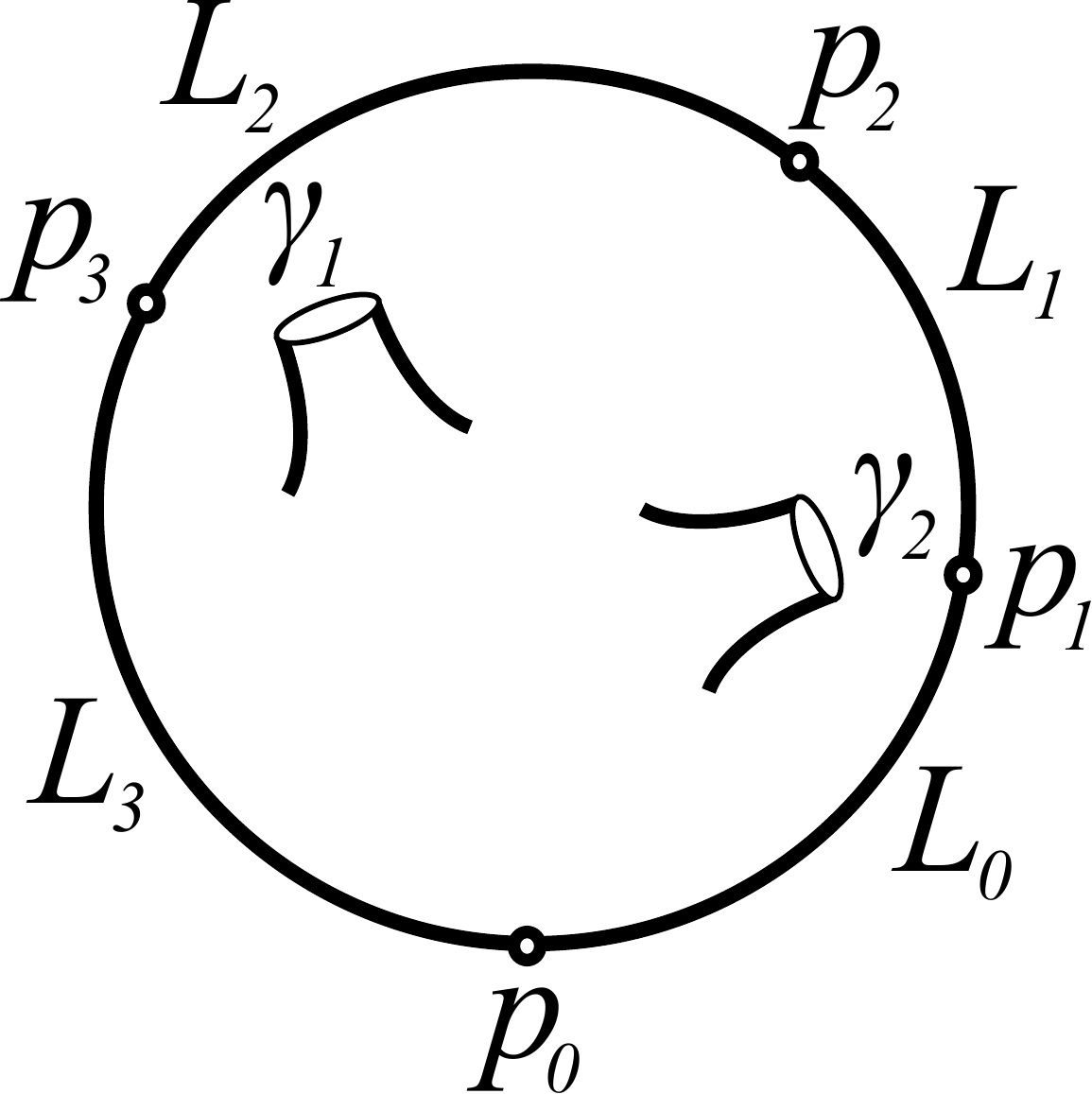}\hfill
\caption{\label{fig:linf} On the left, the domain of a pseudoholomorphic map defining the $L_\infty$ structure on $CF^*(X)$. 
We have omitted some important information from the diagram, which is the `starting point $t=0$' for each of the parametrized orbits $\gamma_i$; by keeping track of this information one can see that the algebraic operations on $CF^*(X)$ should in fact be controlled by chains on the framed little discs operad (which contains the $L_\infty$ operations as a sub-operad). 
On the right, we have illustrated the domain of a pseudoholomorphic map defining the $L_\infty$ homomorphism $\EuC\EuO: CF^*(X) \dashrightarrow CC^*(\fuk(X))$. 
Specifically, we have shown the part of the Hochschild cochain $\EuC\EuO^2(\gamma_1,\gamma_2)$ which sends $p_1 \otimes p_2 \otimes p_3 \mapsto p_0$.}
\end{center}
\end{figure}

There is an $L_\infty$ homomorphism, called the \emph{closed--open map}
\[ \EuC\EuO: CF^*(X) \dashrightarrow CC^*(\fuk(X)),\]
which is defined by counting (pseudo-)holomorphic maps with domain as in Figure \ref{fig:linf} (compare \cite[Theorem Y]{fooo}). 
This is mirror to the map $I^*$ from the previous section. 
We recall that the differential graded Lie algebra of Hochschild cochains controls the curved deformations of $\fuk(X)$ (or equivalently, $\Dfuk(X)$). 
Therefore, if $\EuC\EuO$ is a quasi-isomorphism, then the $L_\infty$ algebra $CF^*(X)$ controls the curved deformations of the Fukaya category by Theorem \ref{thm:mcqi}. 
In fact there is a general criterion for this to be the case:

\begin{thm} 
\label{thm:coiso}
(see \cite[Theorem 6]{Ganatra2016}, and \cite{Ganatra2015,Ganatra2013}) 
If $X$ is compact Calabi--Yau and $\Dfuk(X)$ is homologically smooth, then $\EuC\EuO$ is a quasi-isomorphism.
\end{thm}
 
\begin{rmk}
Roughly, a category is said to be homologically smooth if it admits a categorical `resolution of the diagonal'. It is certainly not clear when $\Dfuk(X)$ should be homologically smooth: for example, symplectic manifolds such as the Kodaira--Thurston manifold have very few interesting Lagrangian submanifolds, and one might expect that they `don't have enough Lagrangians to resolve the diagonal with'. The most efficient way to check homological smoothness of the Fukaya category seems to be via homological mirror symmetry, as we will see in the next section.
\end{rmk}

Besides constructing the differential $\ell^1$ in the $L_\infty$ structure on $CF^*(X)$, Floer also proved that its cohomology is $HF^*(X) \cong H^*(X;\Lambda)$ \cite{Floer1989a} (see also \cite{Piunikhin1996}). 
Morally, the Floer homology `localizes at constant loops': in other words, the inclusion $X \hookrightarrow \mathcal{L} X$ `induces an isomorphism on cohomology'.

In fact it is not difficult to show, using the philosophy `loop rotation acts trivially on the constant loops', that the $L_\infty$ algebra $CF^*(X)$ is homotopy abelian (compare \cite[Theorem 5.1]{Fabert2013}).\footnote{This rather simple observation is `mirror' to the much deeper Kontsevich formality theorem \cite{Kontsevich2003}.}
Therefore, when $\EuC\EuO$ is a quasi-isomorphism the versal deformation space of $\fuk(X)$ is a formal neighbourhood of $0$ in the vector space 
\[ HH^2(\Dfuk(X)) \cong H^2(X;\Lambda).\]
Thus in this situation, all deformations of the Fukaya category come from deformations of the symplectic form. 

\begin{rmk}
\label{rmk:kahlercomp}
The decomposition
\[ H^2(X) = H^{2,0}(X) \oplus H^{1,1}(X) \oplus H^{0,2}(X)\]
corresponds, under mirror symmetry of Hodge diamonds combined with Serre duality, to the decomposition
\[ HT^2(X^\circ) = H^0(\wedge^2 TX^\circ) \oplus H^1(TX^\circ) \oplus H^2(\mathcal{O}_{X^\circ}).\] 
Thus we see that deformations of $\Dcoh(X^\circ)$ `coming from complex deformations of $X^\circ$' correspond to deformations of $\Dfuk(X)$ `coming from deformations of the symplectic form that remain K\"ahler'. 
\end{rmk}

\begin{rmk}
One can study which unobstructed objects $L$ become obstructed if we deform $\fuk\left(X,\omega^\C\right)$ by deforming $\left[\omega^\C\right]$ in the direction of $\eta \in H^2(X;\Lambda)$. 
There is a hierarchy of obstruction classes \cite[Theorem C]{fooo} that must vanish if $L$ is to remain unobstructed, and the leading one can be identified with $\eta|_L \in H^2(L;\Lambda)$. 
This reflects the fact that, when we deform $\omega$ by adding a small closed $2$-form representing $\eta$, we can geometrically deform $L$ so that it remains Lagrangian if and only if $\eta|_L$ is exact.
\end{rmk}

\subsection{HMS and versality}

Let $X$ be a Calabi--Yau K\"ahler manifold with complexified K\"ahler form $\omega^\C$, and $X^\circ$ a smooth projective Calabi--Yau variety over $\Lambda$. 

\begin{defn}
We say that $\left(X,\omega^\C\right)$ and $X^\circ$ are \emph{homologically mirror} if there is a quasi-equivalence of $\Lambda$-linear $A_\infty$ categories
\[ \Dfuk\left(X,\omega^\C\right)^\bc \simeq \Dcoh(X^\circ).\]
\end{defn}

Note that $X^\circ$ has to be defined over $\Lambda$ for the definition to make sense: we can only make sense of the Fukaya category as a $\Lambda$-linear category, so HMS can only identify it with another such. 
In other words, we can only define the Fukaya category over certain $\Lambda$-points of $\mhatkah(X)$; so we can only make a precise statement of HMS which identifies it with the derived category of the fibre over the corresponding $\Lambda$-point of $\mhatcpx(X^\circ)$. 

If $\left(X,\omega^\C\right)$ and $X^\circ$ are homologically mirror, then $\Dfuk\left(X,\omega^\C\right)$ is homologically smooth because $\Dcoh(X^\circ)$ is; so Theorem \ref{thm:coiso} implies that the closed--open map is a quasi-isomorphism. 
Therefore, Strategy \ref{strat:vers} and the previous sections allow us to identify a formal neighbourhood of $0 \in H^2(X;\Lambda)$ with a formal neighbourhood of $0 \in HT^2(X^\circ)$, in a way that matches up the corresponding deformations of $\Dfuk(X,\omega^\C)$ and $\Dcoh(X^\circ)$.

Now let us summarize some unsatisfactory aspects of this implementation of Strategy \ref{strat:vers}. 
\begin{itemize}
\item We had hoped that versality would allow us to conclude something about mirror symmetry for complexified K\"ahler forms nearby $\omega^\C$. 
But the only $\Lambda$-point in a formal neighbourhood of $0 \in H^2(X;\Lambda)$ is $0$ itself, so the deformed Fukaya category doesn't `know about' any $\Lambda$-linear categories other than the original $\Dfuk(X,\omega^\C)$. In particular, versality does not allow us to upgrade to `HMS for complexified K\"ahler forms close to $\omega^\C$'.
\item It's not clear that the identification of versal deformation spaces matches up the deformations of complex structure with the K\"ahler deformations: i.e., it's not clear how to define the Hodge decomposition of Remark \ref{rmk:kahlercomp} in categorical terms. 
\item The deformed categories might have too few unobstructed objects to be `interesting': e.g., they might have no unobstructed objects at all, or they might have too few unobstructed objects to be homologically smooth.
\end{itemize}

Let us finish by expanding somewhat on the first point. 
The universal family of curved deformations of $\Dfuk(X,\omega^\C)$ over $R_v = \Lambda\power{H^2(X;\Lambda)}$ is closely related to the `bulk-deformed Fukaya category' \cite{fooo}, which is defined over $S = \Lambda \otimes_{\Lambda_0} \Lambda_0\power{H^2(X;\Lambda)}$. 
Indeed one can show that the universal family is the pullback of the bulk-deformed Fukaya category via the map induced by the inclusion $S \hookrightarrow R_v$.

We observe that $\spec (S)$ is a larger analytic domain than $\spec (R_v)$: it has $\Lambda$-points corresponding to all $p \in H^2(X;\Lambda_{>0})$, none of which lift to $\spec(R_v)$ except for $0$. 
Thus the bulk-deformed Fukaya category can be specialized to such non-trivial $\Lambda$-points, in contrast to the universal family of curved deformations.
The fibre over such a point corresponds to the Fukaya category equipped with a family of K\"ahler forms parametrized by $q \in \Lambda$, which is equal to $\omega^\C_q$ plus terms of positive valuation. 
However note that we can't get from $\omega^\C_q$ to a different $\eta_q^\C$ by adding elements with positive valuation, because $\log(q) \notin \Lambda_{>0}$. 
Therefore, even the bulk-deformed Fukaya category doesn't `know about' the Fukaya category for nearby K\"ahler forms. 

One could hope to extend the domain of definition of the Fukaya category to an even larger analytic domain than $\spec(S)$ which does `know about' the Fukaya category for nearby K\"ahler forms, using a variation of the `Fukaya trick' \cite{Fukaya:cyclic}. 
However this does not appear straightforward.

Nevertheless, supposing one could extend the parameter space of the Fukaya category in this way, one could furthermore hope to extend the domain over which HMS holds using a standard tool in deformation theory known as Artin approximation \cite{Artin}: this allows one to upgrade a formal solution of a system of analytic equations to an analytic solution. 
This can be applied to the system of analytic equations that must be satisfied by the components of an $A_\infty$ functor realizing homological mirror symmetry. 
However we note that there are infinitely many equations involved in the definition of an $A_\infty$ functor, whose solutions may have radii of convergence tending to $0$. 
Furthermore, even if it were possible to construct an $A_\infty$ functor with a finite radius of convergence, the domain of convergence may still be contained entirely within the locus corresponding to bulk deformations of the Fukaya category. 
In this case we would have a proof of HMS for certain bulk-deformed Fukaya categories, but not for any nearby K\"ahler forms.

\section{Versality at the boundary}
\label{sec:versbound}

\subsection{Versal deformation space at the large complex structure limit}

Let's consider the maximally degenerate quartic surface $\{z_1z_2z_3z_4 = 0\} \subset \CP^3$, which is isomorphic to four copies of $\CP^2$ glued along six lines. 
We will study its versal deformation space $\cM_v$.

One way to deform this variety is to deform the defining quartic equation: this part of the versal deformation space is called the \emph{smoothing locus} $\cM_{sm} \subset \cM_v$.
There also exist locally trivial deformations, which arise by deforming the gluing maps between the six lines along which the four planes are glued: this part of the versal deformation space is called the \emph{regluing locus} $\cM_{rg} \subset \cM_v$.

The existence of these two parts of the versal deformation space is a general phenomenon. 
In fact, the versal deformation space of a `$d$-semistable' maximally degenerate K3 surface is precisely $\cM_v = \cM_{sm} \cup \cM_{rg}$, and these two irreducible components meet transversely \cite[Theorem 5.10]{Friedman1983}.\footnote{We should clarify that the maximally degenerate quartic surface we have used as an example is not $d$-semistable, but becomes so after blowing up an appropriate collection of $24$ singular points \cite[Remark 1.14]{Friedman1983}}
In particular we cannot expect a versality criterion as simple as Theorem \ref{thm:vers} in this case. 

We are really only interested in the smoothing locus $\cM_{sm}$ (unless we want to study HMS for degenerate varieties). 
So we would like to modify our deformation problem so that the versal deformation space consists only of the smoothing locus. 
One way to do this is provided by \cite{Kawamata1994}, which shows that the \emph{logarithmic} versal deformation space of a $d$-semistable $K3$ surface is smooth (the regluings do not preserve the log structure so they are ruled out).
This gives us a chance to implement Strategy \ref{strat:vers}, if we only knew what logarithmic deformations corresponded to on the Fukaya category. 
Unfortunately I have no idea, so instead we will use a different modification of the deformation problem, motivated by the other side of mirror symmetry.

\subsection{Versal deformation space at the large volume limit}

We study the versal deformation space at the large volume limit, following Seidel \cite{Seidel2002,Seidel2003}.

Let $X$ be a Calabi--Yau K\"ahler manifold and $D \subset X$ be a simple normal-crossings divisor supporting an effective ample divisor as in \S \ref{subsubsec:all}. 
We have seen in \S \ref{subsubsec:atlvl} that the relative Fukaya category $\fuk(X,D)$ is a deformation of the affine Fukaya category $\fuk(X \setminus D)$ over $R_{X,D}$. 
Thus we would like to study the curved deformations of $\fuk(X \setminus D)$, which are controlled by $CC^*(\fuk(X \setminus D))$. 
We denote the versal deformation space of this category by $\cM_v$ again. 

The relative Fukaya category gives one deformation, classified by a map $\mbarhatkah(X,D) \to \cM_v$. 
One can obtain others by applying birational modifications to $X$ along $D$ to obtain a new compactification $(X',D')$ with $X' \setminus D' = X \setminus D$: these all fit together into an extended relative K\"ahler moduli space, which we call the \emph{compactification locus} $\cM_{comp}$ (see \cite[\S 6.2.2]{coxkatz} for a discussion). 
So we have a map $\cM_{comp} \to \cM_v$, through which all of the classifying maps from $\mbarhatkah(X',D')$ factor. 

However there's another way to deform the category: by deforming the exact symplectic form $\omega|_{X \setminus D}$ to a non-exact one, by analogy with \S \ref{subsec:HF}. 
Such deformations are parametrized by a formal neighbourhood of $0 \in H^2(X \setminus D)$, which we call the \emph{non-exact locus} $\cM_{ne}$. 
I expect that the compactification locus is mirror to the smoothing locus, and the non-exact locus is mirror to the regluing locus (this was initially suggested to me by Helge Ruddat). 

Analogously to \S \ref{subsec:HF}, there exists an $L_\infty$ algebra $SC^*(X \setminus D)$ which one expects should come with an $L_\infty$ morphism
\begin{equation}
\label{eqn:COSH}
 \EuC\EuO: SC^*(X \setminus D) \dashrightarrow CC^*(\fuk(X \setminus D))
\end{equation}
(see \cite{Fabert2013}), which can be hoped to be a quasi-isomorphism by analogy with work of Ganatra \cite{Ganatra2013}. 
The cochain complex $(SC^*(X \setminus D),\ell^1)$ was introduced in \cite{Floer1994} and its cohomology is called the \emph{symplectic cohomology} $SH^*(X \setminus D)$ (references include \cite{Viterbo1999,Oancea2004,Seidel2008a}). 

Once again, $SH^*(X \setminus D)$ is the semi-infinite Morse cohomology of an action functional $\mathcal{A}_H$ on the free loop space of $X$. 
However because $X \setminus D$ is non-compact, there are various growth conditions one can put on the Hamiltonian $H_t$ at infinity, giving rise to different versions of symplectic cohomology. 
The standard condition, and the one for which \eqref{eqn:COSH} turns out to be a quasi-isomorphism in examples, is that $H_t$ should `go to $+\infty$ sufficiently rapidly near $D$'. 
As a consequence, $SH^*(X \setminus D)$ no longer `localizes at constant loops', because there will exist orbits $\gamma$ of $X_{H_t}$ linking $D$.  
In fact there is a short exact sequence of cochain complexes
\begin{equation}
\label{eqn:ses}
 0 \to C^*(X \setminus D) \to SC^*(X \setminus D) \to SC^*_+(X \setminus D) \to 0,
 \end{equation}
where $C^*(X \setminus D)$ corresponds to the `constant loops' and $SC^*_+(X \setminus D)$ to loops `linking $D$' \cite{Bourgeois2009a}.

Let us assume that the closed--open map is an isomorphism, so $HH^2(\fuk(X \setminus D)) \cong SH^2(X \setminus D)$. 
The short exact sequence \eqref{eqn:ses} induces a long exact sequence
\begin{equation}
\label{eqn:les} \ldots \to H^2(X \setminus D) \to SH^2(X \setminus D) \to SH^2_+(X \setminus D) \to \ldots .
\end{equation}
We have seen that the versal deformation space sits inside a formal neighbourhood of $0 \in HH^2(\fuk(X \setminus D)) \cong SH^2(X \setminus D)$, so its tangent space is $T\cM_v \cong SH^2(X \setminus D)$. 

This allows us to see the two components of our deformation space we mentioned above. 
The derivative of the map $\cM_{ne} \to \cM_v$ is identified with the inclusion $H^2(X \setminus D) \hookrightarrow SH^2(X \setminus D)$. 
The derivative of the map $\cM_{comp} \to \cM_v$ can be composed with a map from the long exact sequence \eqref{eqn:les} to give 
\begin{equation}
\label{eqn:derivkah}
 T \cM_{comp} \to T\cM_v = SH^2(X \setminus D) \to SH^2_+(X \setminus D),
\end{equation}
which can be interpreted as follows. 
Firstly, under our assumption that $X$ is Calabi--Yau and some mild hypotheses on $D$, $SH^2_+(X \setminus D)$ has one generator for each irreducible component of $D$, corresponding to an orbit $\gamma_p$ linking that component (compare, e.g., \cite{Ganatra2016a}). 
Deforming the relative K\"ahler class in the direction of $PD(D_p)$ corresponds, under the map \eqref{eqn:derivkah}, to moving in the direction of $\gamma_p$ (Figure \ref{fig:div} is supposed to suggest why this is true).

\begin{figure}
\begin{center}
\includegraphics{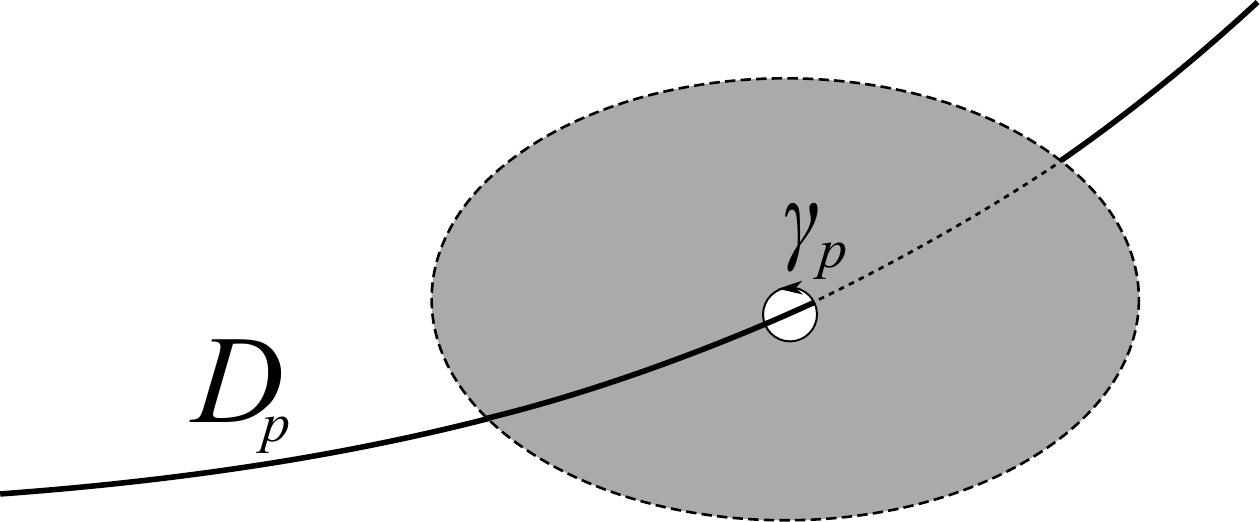}
\caption{\label{fig:div} Holomorphic maps (in grey) with an asymptotic condition along the orbit $\gamma_p$ linking $D_p$, correspond to holomorphic maps with an intersection point with $D_p$.}
\end{center}
\end{figure}

Thus we see that the tangent spaces of $\cM_{comp}$ and $\cM_{ne}$ span the tangent space of $\cM_v$ at the origin. 
It could still be the case that the versal deformation space is smooth, but we can give a heuristic argument indicating that this is not the case in general. 
Namely, if we deform along $\cM_{ne}$ to a K\"ahler form $\eta$ whose cohomology class $[\eta] \in H^2(X \setminus D)$ does not lie in the image of the restriction map $H^2(X) \to H^2(X \setminus D)$, then we can no longer deform in the compactification direction: so the dimension of the tangent space drops when we move in this direction. 

\subsection{Cutting down the deformations}
\label{subsec:cut}

We have seen that our respective versal deformation spaces have additional components (the regluing and non-exact loci, respectively), which are an obstruction to applying Theorem \ref{thm:ainfvers}. 
One solution to this problem would be to prove a better versality theorem; another would be to eliminate the extra components by modifying our deformation problem. 
No doubt there's more than one way to do this, but we will choose to do it by imposing an additional equivariance constraint (as was done in \cite{Seidel2003}, although our equivariance will be with respect to anti-symplectic rather than symplectic symmetries).
 
Namely, we suppose that $(X,D)$ is equipped with a real structure. 
This induces an anti-holomorphic involution $\iota: X \to X$, preserving $D$ as a set. 
It acts on cohomology by
\begin{align*}
\iota^*: H^*(X) & \to H^*(X), \quad \text{ sending} \\
PD(D_p) & \mapsto -PD(D_p)
\end{align*}
for any component $D_p$ of $D$ (because it reverses orientation on the normal bundle to $D_p$). 
In particular it acts on the cohomology class of any relative K\"ahler form $[\omega^\C;\theta^\C]$ by $-1$. 

Let us suppose that $\iota^*: H^2(X \setminus D) \to H^2(X \setminus D)$ acts by the identity. 
Then it acts by $+1$ on the tangent space to $\cM_{ne}$, but $-1$ on the tangent space to $\cM_{comp}$. 
Therefore if we restrict to `$\iota$-anti-invariant' deformations in the appropriate sense, the versal deformation space will consist only of the desired component $\cM_{comp}$.

\begin{example}
\label{eg:Batdiv}
Let $X \subset V$ be a hypersurface in a toric variety as in Batyrev's mirror construction, and let $D \subset X$ be its transverse intersection with the toric boundary divisor $\partial V$. 
If the defining equation of $X$ is real, then the real structure on $(V,\partial V)$ descends to one on $(X,D)$. 
Complex conjugation acts by the identity on $H^2(V \setminus \partial V)$, so if the restriction map
\[ H^2(V \setminus \partial V) \to H^2(X \setminus D)\]
is surjective, then it also acts by the identity on $H^2(X \setminus D)$. 
This is the case as soon as $X$ has dimension $\ge 3$, by the affine Lefschetz hyperplane theorem \cite[Theorem 6.5]{Dimca1992}.
\end{example}

Whereas a holomorphic involution induces an involution on the Fukaya category (an endofunctor squaring to the identity), an \emph{anti-}holomorphic involution $\iota$ induces a functor from the Fukaya category to its \emph{opposite} category:
\begin{equation}
\label{eqn:duality}
\mathcal{D}: \fuk(X \setminus D) \xrightarrow{\sim} \fuk(X \setminus D)^{op},
\end{equation}
satisfying $\mathcal{D}^{op} \circ \mathcal{D} = \id$. 
This observation is due to \cite{Castano2010}, where such functors are called \emph{dualities}. 
The name is motivated by the other side of mirror symmetry, where a natural duality is the functor
\begin{equation}
\label{eqn:bduality}
\mathsf{R} \Hom(-,\mathcal{O}_{X^\circ}): \Dcoh (X^\circ) \xrightarrow{\sim} \Dcoh (X^\circ)^{op} 
\end{equation}

The duality $\mathcal{D}$ manifestly extends to the relative Fukaya category, so we modify our deformation problem to consider only such deformations. 
This modified deformation problem is controlled by a sub-dg Lie algebra of $CC^*(\fuk(X \setminus D))$, which is the invariant subspace of the involution induced by the functor \eqref{eqn:duality}. 

With this modification, there is a chance that the versal deformation space at the large volume limit point is realized by the compactification locus. 
Therefore we have a chance to implement Strategy \ref{strat:vers}, which is what we will do in the next section. 

\begin{rmk}
Observe that any relative K\"ahler form has $[\omega^\C] = \sum_k \lambda_k \cdot PD(D_k) \in H^{1,1}(X)$. 
Comparing with Remark \ref{rmk:kahlercomp}, we notice another advantage of our current setup: by moving within the compactification locus we avoid non-K\"ahler deformations of the symplectic form, which correspond to non-commutative deformations of the mirror.
\end{rmk}

\section{Proving mirror symmetry}
\label{sec:provehms}

In this section we summarize the implementation of Strategy \ref{strat:vers} suggested by the arguments of \S \ref{sec:versbound}. 

\subsection{Setup}

Let $X$ be Calabi--Yau K\"ahler, $D = \cup_{p \in P} D_p \subset X$ simple normal-crossings such that $(X,D)$ admits a real structure, and $(\omega^\C,\theta^\C)$ a complexified relative K\"ahler form. 
Recall that we denote 
\[\mbarhatkah(X,D) := \spec(R_{X,D}),\]
 and the complexified relative K\"ahler form determines a $\Lambda$-point $p$ of $\mbarhatkah(X,D)$.

Let $\EuX^\circ$ be a scheme over $R := \C\power{\nov_1,\ldots,\nov_k}$. 
We denote 
\[\mbarhatcpx(X^\circ) := \spec(R),\]
and let $X^\circ_p$ denote the fibre of $\EuX^\circ$ over a $\Lambda$-point $p$ of $\mbarhatcpx(X^\circ)$.

Suppose that we are trying to prove a version of HMS of the form:

\begin{conj}
\label{conj:hmsvers}
There exists a map $\Psi: \mbarhatkah(X,D) \to \mbarhatcpx(X^\circ)$ such that there is a quasi-equivalence
\[  \Dcoh(X^\circ_{\Psi(p)}) \simeq \Dfuk\left(X,\omega^\C\right)^\bc.\]
\end{conj}

\begin{example}
\label{eg:Batvers}
Batyrev mirror symmetry fits into this setup (Example \ref{eg:Batyrev}). 
In this case $X$ is a hypersurface in a toric variety $V$, and $D$ is the intersection with the toric boundary. 
The components of $D$ can be indexed by a certain set $\Xi$ of boundary lattice points of the reflexive polytope $\Delta^\circ$, so $R_{X,D}$ is a power series ring in certain monomials in variables $\{\nov_k\}_{k \in \Xi}$ (more precisely this is the `ambient' coordinate ring, compare \cite{Sheridan2017a}). 
On the mirror, the lattice points $k \in \Delta^\circ$ index a monomial basis $\{\chi^k\}$ of the anticanonical bundle of the mirror toric variety $V^\circ$, so we can take $\EuX^\circ$ to be the family of hypersurfaces in $V^\circ$ defined by
\[ \chi^0 = \sum_{k \in \Xi} \nov_k \cdot \chi^k.\]
It is defined over $R := \C\power{\nov_k}_{k \in \Xi}$. 
\end{example}

\subsection{Applying versality at the boundary}
\label{subsec:applyvers}

Let $\cA \subset \fuk(X,D)$ be a subcategory of the relative Fukaya category, preserved by the duality $\mathcal{D} : \fuk(X,D) \to \fuk(X,D)^{op}$ induced by the real structure (so $\cA$ has a duality). 

Let $\cB \subset \Dcoh(\EuX^\circ)$ be a subcategory of an appropriate DG enhancement for the derived category, preserved by the duality \eqref{eqn:bduality}. 

We now have an $R_{X,D}$-linear $A_\infty$ category $\cA$, and an $R$-linear DG category $\cB$, both equipped with dualities. 
We denote the respective quotients by the maximal ideals by $\cA_0$ and $\cB_0$. 
We will suppose that we have `proved HMS at the limit points': there is an $A_\infty$ isomorphism $\cB_0 \dashrightarrow \cA_0$ which preserves dualities.\footnote{To the reader who smells a rat already, we'll come clean in Remarks \ref{rmk:rat1} and \ref{rmk:rat2}.}

Secondly we will assume that `$\cB$ is versal among duality-preserving deformations': if $\fg \subset CC^*(\cB_0)$ is the invariant subspace of the involution induced by the duality, then the Kodaira--Spencer map associated to the deformation $\cB$
\begin{equation}
\label{eqn:KSB}
 KS_\cB: (\fm/\fm^2)^* \to H^1(\fg)
\end{equation}
is an isomorphism. 

Under these assumptions, we have:

\begin{thm}
\label{thm:msvers}
There exists a map $\Psi^*: R \to R_{X,D}$ together with a quasi-embedding of $A_\infty$ categories 
\[\Psi^*\cB \hookrightarrow \cA^\bc.\] 
\end{thm}
\begin{proof}
By Theorem \ref{thm:ainfvers} (modified to take the dualities into account), we have a quasi-embedding
\[ \Psi^*\cB^\bc \hookrightarrow \cA^\bc.\]
Since $\Psi^*\cB$ is uncurved, it is a full subcategory of $\Psi^*\cB^\bc$ (corresponding to equipping each object with the zero bounding cochain). 
Thus we have an embedding $\Psi^* \cB \hookrightarrow \cA^\bc$ as required.
\end{proof}

\begin{rmk}
\label{rmk:rat1}
The hypothesis that $\cA_0$ and $\cB_0$ are $A_\infty$ isomorphic (rather than quasi-equivalent), which is imposed by Theorem \ref{thm:ainfvers} (see Remark \ref{rmk:ainfiso}), will never, ever be satisfied. 
The reason is that morphism spaces in the Fukaya category can be chosen to be finite-dimensional on the cochain level, whereas those in the derived category are typically infinite-dimensional (although they may have finite-dimensional cohomology). 
One can fix this by passing to quasi-isomorphic models for $\cA$ and $\cB$, although one needs to be a little careful about what one means since $A_\infty$ quasi-isomorphisms over a ring like $R$ need not be invertible up to homotopy, and minimal models need not exist (in contrast to the case over a field). 
In examples, one can check the existence of appropriate minimal models using ad-hoc methods \cite{Sheridan2015,Sheridan2017a}.
\end{rmk}

\begin{rmk}
\label{rmk:rat2}
Identifying $\cA_0$ with $\cB_0$ is only part of `proving HMS at the limit points'. 
However it's sufficient for our purposes (compare \S \ref{sec:skel}).
\end{rmk}

\subsection{Automatic generation}
\label{subsec:gen}

The map $\Psi^*$ provided by Theorem \ref{thm:msvers} can be interpreted as the pullback of functions via the mirror map 
\[\Psi: \mbarhatkah(X,D) \to \mbarhatcpx(X^\circ).\] 
Taking fibres, we obtain a quasi-embedding $\cB_{\Psi(p)} \hookrightarrow \cA_p^\bc$. 
This allows us to identify quasi-embeddings
\[\xymatrix{ \Dcoh(X^\circ_{\Psi(p)}) & \Dfuk\left(X,\omega^\C\right)^\bc \\
\cB_{\Psi(p)} \ar@{^{(}->}[u] \ar@{^{(}->}[r]^\sim & \cA_p^\bc. \ar@{^{(}->}[u]}\] 
We now seek to show that these subcategories (split-)generate, so that the identification extends to the categories upstairs (see \cite[\S 4]{Seidel2008}). 
Abouzaid gives an extremely useful criterion for a subcategory to split-generate the Fukaya category \cite{Abouzaid2010a}, and there are results showing that his criterion is automatically satisfied under certain conditions, which dramatically reduce the amount of work required \cite{Perutz2015,Ganatra2016}:

\begin{thm} (Ganatra \cite{Ganatra2016}) 
\label{thm:autgen}
If $\cB_{\Psi(p)}$ split-generates $\Dcoh(X^\circ_{\Psi(p)})$, then its image split-generates $\Dfuk\left(X,\omega^\C\right)^\bc$. 
In particular we have a quasi-equivalence
\[ \Dcoh\left(X^\circ_{\Psi(p)}\right) \simeq \Dfuk\left(X,\omega^\C\right)^\bc,\]
so $\left(X,\omega^\C\right)$ and $X^\circ_{\Psi(p)}$ are homologically mirror.
\end{thm}

In the case of Batyrev mirrors (Example \ref{eg:Batvers}) we might, for example, choose $\cB$ to be the full subcategory of derived restrictions of some split-generators of $\Dcoh(V^\circ)$. 
These split-generate whenever $X^\circ_{\Psi(p)}$ is smooth (see \cite[Lemma 5.4]{Seidel2003}), so in order to apply Theorem \ref{thm:autgen} in this case it would suffice for us to check this smoothness.
In order to determine something about the smoothness of $X^\circ_{\Psi(p)}$, we need to know something about $\Psi$: for example, if $\Psi$ sent all of $\mbarhatkah(X,D)$ to the large volume limit point then $X^\circ_{\Psi(p)}$ would never be smooth so we would never be able to apply Theorem \ref{thm:autgen}. 

We can gain some control over $\Psi$ by making a versality-type assumption on the Fukaya category: a natural such assumption is that all of the maps
\begin{equation}
\label{eqn:comm}
\xymatrix{SH^2_+(X \setminus D) \ar[r]^-{\EuC\EuO} & H^1(\fg) \ar@{<->}[d]^-\cong \\
\left(\fm/\fm^2\right)^* \ar[r]^-{KS_\cB} & H^1(\fg)}
\end{equation} 
are isomorphisms. 
Although this rules out pathologies like $\Psi$ being constant, in general it is still not strong enough to allow us to guarantee that $X^\circ_{\Psi(p)}$ is smooth, as the following example shows.

\begin{example}
\label{eg:GGP}
The strategy we describe has been implemented in \cite{Sheridan2017a} to prove HMS (in the form of Conjecture \ref{conj:hmsvers}) for Batyrev mirrors in the case that $\Delta$ is a simplex (also known as Greene--Plesser mirrors \cite{Greene1990}), as well as certain generalizations within the Batyrev--Borisov framework \cite{Batyrev1994}. 
It turned out that, in order to get sufficient control over $\Psi$ to establish smoothness of $X^\circ_{\Psi(p)}$, Theorem \ref{thm:msvers} was not quite strong enough. 
After establishing that the maps in \eqref{eqn:comm} were all isomorphisms, we had to apply the more elaborate versality theorem proved in \cite{Sheridan2017}, which takes advantage of certain extra structures present in this case (the algebraic torus action on the toric variety $V^\circ$ containing $X^\circ$, which is mirror to the `anchored Lagrangian Floer theory' in $(X,D)$ \cite{Fukaya2010a}, some convenient facts about the shape of the K\"ahler cone of the toric variety $V$ containing $X$, and a certain finite symmetry group of $(X,D)$). 
This gave us sufficient information about the valuations of $\Psi^*(\nov_k)$ to establish smoothness of $X^\circ_{\Psi(p)}$ using the mirror relationship between the K\"ahler cone of $X$ and the secondary fan associated to $\Xi$ (see \cite[\S 6]{coxkatz} for this relationship, and \cite{Gelfand1994,Oda1991,Dickenstein2007} for background).
\end{example}

\begin{rmk}
Example \ref{eg:GGP} explains one reason why the more elaborate versality criterion of \cite{Sheridan2017} is necessary: it gives us important extra control over the mirror map $\Psi$. 
Another reason is that it applies even if the complex moduli space is not smooth at the large complex structure limit point. 
This smoothness was built into our assumptions, since we assumed $R$ was a formal power series ring, but in general the large complex structure limit point can be more complicated.  
\end{rmk}

\begin{rmk}
Although we assume $X$ to be Calabi--Yau in this paper, HMS also makes sense if $X$ is Fano \cite{Kontsevich1998,Seidel2001,Seidel2001b} or of general type \cite{Kapustin2010,Seidel2008a}. 
The versality criterion of \cite{Sheridan2017} assumes that $-K_X$ is nef, and in particular can be applied in Calabi--Yau or Fano situations.
\end{rmk}

\subsection{Closed-string mirror symmetry}
\label{subsec:hmsimplies}

It remains to make the connection back to closed-string mirror symmetry. 
Closed-string mirror symmetry can be formulated in this setting as an isomorphism of variations of Hodge structures $V^A(X) \cong \Psi^* V^B(X^\circ)$, where $\Psi: \mhatkah(X,D) \to \mhatcpx(X^\circ)$ is the mirror map. 
One of the features of closed-string mirror symmetry is that the variations of Hodge structures determine `canonical' or `flat' coordinates on the respective moduli spaces: roughly, this is something like an affine structure (the procedure for determining these coordinates is due to Saito \cite{Saito1983}, see \cite{coxkatz} for an explanation in the setting of mirror symmetry). 
If closed-string mirror symmetry holds, it follows that $\Psi$ must match up these canonical coordinates, and this determines $\Psi$ uniquely up to a linear change of variables. 
This linear change of variables is fixed by knowing the derivative of $\Psi$ at the origin. 

In the previous sections we described a set of hypotheses (including `HMS holds at the large volume/large complex structure limit') and how they should imply HMS in the form of Conjecture \ref{conj:hmsvers}. 
While it is not the case that this version of HMS implies closed-string mirror symmetry in any useful sense (at least as far as I understand), in this section we explain how the same set of hypotheses leads to a proof of closed-string mirror symmetry involving the same mirror map $\Psi$. 
As discussed above, this determines $\Psi$ up to a linear change of variables which is fixed by knowing the derivative of $\Psi$ at the origin. 
This derivative corresponds to the composition of isomorphisms \eqref{eqn:comm}, so if we can compute that then $\Psi$ is determined uniquely. 

To summarize: versality gives us a mirror map $\Psi$ whose derivative at the origin is determined by \eqref{eqn:comm}, but does not give any information about its higher-order derivatives.
However, these higher-order derivatives turn out to be determined uniquely by Hodge theory.

The basic idea is that, associated to a family of $A_\infty$ categories $\cA$ over $\cM$, there is a variation of Hodge structures $HP_\bullet(\cA)$ over $\cM$ called the \emph{periodic cyclic homology of $\cA$} (the connection was constructed by Getzler \cite{Getzler1993}; see \cite{Katzarkov2008}, \cite{Costello2009} or \cite{Sheridan2015a} for more precise statements). 
It is convenient for us to work exclusively with variations of Hodge structures over the formal punctured disc $\spec(\C\laurents{q})$. 

Let $p$ be a $\C\laurents{q}$-point of $\mhatcpx(X^\circ)$ such that $X^\circ_p$ is smooth. 
Then there is an isomorphism of bundles over the formal punctured disc
\begin{equation}
HP_\bullet(\Dcoh(X^\circ_p)) \cong V^B(X^\circ_p)
\end{equation}
which is conjectured to be an isomorphism of variations of Hodge structures (see \cite[Conjecture 1.14]{Ganatra2015}; as explained there, we believe this is `standard' in the right community but the precise statement has not appeared in the literature). 
Under our previous assumption that $\cB_p$ split-generates $\Dcoh(X^\circ_p)$, we have $HP_\bullet(\Dcoh(X^\circ_p)) \cong HP_\bullet(\cB_p)$ by Morita invariance of periodic cyclic homology.

On the mirror, let us suppose that we have an \emph{integral} complexified relative K\"ahler class: i.e., $[\omega^\C;\theta^\C] \in H^2(X,X \setminus D;\R \oplus \iii \Z)$. 
Then we can define a corresponding $\C \laurents{q}$-point $p$ of $\mhatkah(X,D)$ by \eqref{eqn:lambdapoint}, and there is a homomorphism of variations of Hodge structures called the \emph{cyclic open--closed map}
\begin{equation}
\wt{\EuO\EuC}: HP_\bullet(\fuk(X,D)^\bc_p) \to V^A(X)_p
\end{equation}
(this was announced in \cite{Ganatra2015}, but we believe it to be implicit in \cite{Costello2007}; the morphism of bundles will be constructed in \cite{Ganatraa}, and it will be proved to be a morphism of variations of Hodge structures in \cite{Ganatra2015a}). 
Under our previous assumption that $\cB_{\Psi(p)}$ split-generates $\Dcoh(X^\circ_{\Psi(p)})$, the restriction of $\wt{\EuO\EuC}$ to the subcategory $\cA^\bc_p$ is an isomorphism 
\[ \wt{\EuO\EuC}: HP_\bullet(\cA^\bc_p) \xrightarrow{\sim} V^A(X)_p\]
(this is a consequence of \cite{Ganatra2016} combined with \cite[Theorem 5.2]{Ganatra2015}). 

As a consequence, the embedding $\Psi^*\cB \subset \cA^\bc$ determines isomorphisms $V^A(X)_p \cong V^B(X^\circ)_{\Psi(p)}$ for all such $p$. 
Thus we have proved closed-string mirror symmetry over certain formal punctured discs in the moduli space. 
This is not quite the same as a complete proof of closed-string mirror symmetry, but it serves the same purpose: in particular it determines the mirror map uniquely (see \cite[Appendix C]{Sheridan2017} for precise statements), and it also allows us to extract all of the enumerative information that we want from closed-string mirror symmetry at genus zero.

\begin{rmk}
It is expected or hoped that the higher-genus versions of closed-string mirror symmetry can also be recovered from HMS: see \cite{Kontsevich2003a,Costello2007,Kontsevich2006a,Costello2009,Caldararu2017}.
\end{rmk}

\begin{rmk}
\label{rmk:nofuksneeded}
Note that we did not need the category $\fuk\left(X,\omega^\C\right)$ to prove closed-string mirror symmetry: the relative Fukaya category $\fuk(X,D)$ sufficed.
\end{rmk}

\section{The wrapped Fukaya category}
\label{sec:skel}

\subsection{HMS at the limit points}

In our application of versality in \S \ref{sec:provehms}, our first step when trying to identify the subcategories $\cA \subset \fuk(X,D)$ and $\cB \subset \Dcoh(\EuX^\circ)$ was to identify them `at the limit': $\cA_0 \simeq \cB_0$. 
It might be tempting to conjecture that one can go beyond these subcategories, and identify 
\[ \Dfuk(X \setminus D) \simeq \Dcoh(X^\circ_0).\]
 
However these categories can't be quasi-equivalent: as soon as $X^\circ_0$ is singular, there exist objects (e.g., a skyscraper sheaf at a singular point) whose endomorphism algebra in the derived category is infinite-dimensional.
Such a sheaf can't be mirror to a compact Lagrangian $L \subset X \setminus D$, since a generic perturbation of $L$ off of itself will intersect it in finitely many points so $hom^*_{\fuk(X \setminus D)}(L,L)$ is always finite-dimensional. 

Instead, one expects that HMS at the large volume/large complex structure limit should give quasi-equivalences
\begin{equation}
\label{eqn:hmslvl}
 \xymatrix{ \Dwfuk(X \setminus D) \ar@{<->}[r]^-\sim \ar@{}^\bigcup[d] & \Dcoh(X^\circ_0) \ar@{}^\bigcup[d] \\
\Dfuk(X \setminus D) \ar@{<->}[r]^-\sim & \perf(X^\circ_0).}
\end{equation}
Here, $\perf(X^\circ_0) \subset \Dcoh(X^\circ_0)$ is the subcategory generated by locally-free coherent sheaves (it has finite-dimensional morphism spaces).
The other category we've introduced is the \emph{wrapped Fukaya category} $\wfuk(X \setminus D)$ \cite{Abouzaid2007}. 
This is a category which admits possibly non-compact Lagrangians $L \subset X \setminus D$ (which however are required to go to infinity in a prescribed way), and whose morphism spaces $hom^*_{\wfuk(X \setminus D)}(K,L)$ count intersection points between a $K$ and a `wrapped' version of $L$ (the `wrapping' is by the same Hamiltonian flow whose orbits are the generators of symplectic cohomology). 
It can in particular have infinite-dimensional morphism spaces. 
See \cite{Auroux2013} for an expository account of the wrapped Fukaya category.

One nice thing about the wrapped Fukaya category is that it is expected that it should always have `enough objects' when $X \setminus D$ is an affine variety: it should be split-generated by the ascending manifolds of the middle-dimensional critical points of a plurisubharmonic Morse function \cite{Bourgeois2009a,Ganatra2013}. 
In contrast, there is no such natural source of compact Lagrangians in $X \setminus D$, so we have no guarantee that $\Dfuk(X \setminus D)$ is well-behaved (it might even have an empty set of objects). 

The subcategory $\perf(X^\circ_0) \subset \Dcoh(X^\circ_0)$ can be characterized as the subcategory of \emph{proper} objects $K$, i.e., those such that $hom^*_{\Dcoh(X^\circ_0)}(K,L)$ has finite-dimensional cohomology for all $L$ \cite[Lemma 3.11]{Ballard2011}. 
Any object of the subcategory $\fuk(X \setminus D) \subset \wfuk(X \setminus D)$ also has this property, because a compact Lagrangian $K \subset X \setminus D$ will only intersect $L$ in finitely many points, even if $L$ is non-compact and gets wrapped (this property persists under the formal enlargement procedure denoted `$\mathsf{D}$'). 
However there may exist proper objects which do not come from compact Lagrangians, so \eqref{eqn:hmslvl} will have a better chance of being true if we replace the lower left-hand corner by the category $\Dwfuk(X \setminus D)^{prop}$ of proper objects of the wrapped Fukaya category.

\subsection{Computing the wrapped category}
\label{subsec:wrapcomp}

The wrapped Fukaya category is expected to exhibit `sheafy' behaviour. 
A result in this direction was proved in \cite{Abouzaid2007}, which (besides giving the original definition of the wrapped Fukaya category) established the existence of restriction maps to certain subdomains. 
A conjecture in a rather different direction was made by Kontsevich \cite{Kontsevich2009}.
He considered the `skeleton' of $X \setminus D$, which is the limit of the reverse Liouville flow: it is an isotropic cell complex which is a deformation retract $ \Lambda \subset X \setminus D$ \cite{Biran2001}.
His conjecture said that the wrapped Fukaya category should be the the global sections of some cosheaf of categories on this skeleton, defined in terms of its local geometry.  

This idea was developed by several authors \cite{Abouzaid2011c,Sibilla2014,Pascaleff2016,Sylvan2016}. 
Nadler in particular has made it more precise, by describing a sufficiently generic class of skeleta to work with (the `arboreal' ones), and formulating Kontsevich's conjecture as a quasi-equivalence 
\begin{equation}
\label{eqn:mush}
 \mu Sh_\Lambda(X \setminus D)^w \simeq \Dwfuk(X \setminus D)
 \end{equation}
where the left-hand side is the category of `wrapped microlocal sheaves' supported along such a skeleton $\Lambda$ \cite{Nadler2014,Nadler2017,Nadler2017a,Nadler2015}. 
It has been announced in \cite{Gammage2017} that this equivalence will be established in work-in-preparation of Ganatra--Pardon--Shende, building on the partial results in \cite{Ganatra2017}.

The nice thing about wrapped microlocal sheaves is that they are \emph{computable}. 
In contrast, direct computation of the wrapped Fukaya category requires us to enumerate pseudoholomorphic discs, and there's no systematic way to do that (we made the point in Remark \ref{rmk:enumless} that versality ideas reduce us to a finite number of such enumerations, but we still have to do them somehow!). 
Thus it may be easier to prove 
\begin{equation}
\label{eqn:hmsmush} \mu Sh_\Lambda(X \setminus D)^w \simeq \Dcoh(X^\circ_0),
\end{equation}
which implies HMS at the large volume limit, in the sense of the top line of \eqref{eqn:hmslvl}, if \eqref{eqn:mush} holds.
This has been the approach taken in \cite{Nadler2016,Nadler2016a,Nadler2017,Gammage2017}, where versions of the HMS equivalence \eqref{eqn:hmsmush} have been proved in various cases.

\subsection{Deforming the wrapped category}
\label{subsec:defwrap}

Since the wrapped Fukaya category is so much better-behaved and more computable than the compact Fukaya category $\fuk(X \setminus D)$, it's tempting to try to incorporate it into the strategy outlined in \S \ref{sec:provehms}. 
The problem is that it's not obvious what the analogue of $\fuk(X,D)$ should be. 
Whereas it's easy to deform $\fuk(X \setminus D)$ by counting holomorphic discs passing through $D$, it's not immediately clear how to deal with discs passing through $D$ but having boundary on objects of $\wfuk(X \setminus D)$, which may be non-compact Lagrangians that approach $D$ in a complicated way. 

There is a formal solution to this problem, building on ideas of \cite{Seidel2002,Fabert2014}: in work-in-preparation with Borman \cite{Borman2015,Borman2016}, we construct a Maurer--Cartan element in the $L_\infty$ algebra $SC^*(X \setminus D)$ (under rather restrictive hypotheses on $D$ however: we require each irreducible component to be ample). 
The pushforward of this Maurer--Cartan element via the closed--open map 
\[ \EuC\EuO: SC^*(X \setminus D) \dashrightarrow CC^*(\fuk(X \setminus D))\]
should be gauge equivalent to the Maurer--Cartan element corresponding to the deformation $\fuk(X,D)$ of $\fuk(X \setminus D)$. 
The closed--open map should extend to an $L_\infty$ homomorphism 
\[\EuC\EuO: SC^*(X \setminus D) \dashrightarrow CC^*(\Dwfuk(X \setminus D))\] 
(the leading-order term is constructed in \cite{Ganatra2013}). 
It follows that we can define a curved $A_\infty$ category by deforming the wrapped category by the image of the Maurer--Cartan element under this morphism. 
Passing to bounding cochains, then taking the fibre over the appropriate $\Lambda$-point $p$, we obtain a $\Lambda$-linear $A_\infty$ category which we denote by $\Dwfuk(X,D)^\bc_p$.

This leads us to the following idea: perhaps there is a `better' Fukaya category $\tilde{\fuk}\left(X,\omega^\C\right)$, such that:
\begin{itemize}
\item $\Dfuk\left(X,\omega^\C\right)^\bc \subset \tilde{\fuk}\left(X,\omega^\C\right)$ is a full subcategory;
\item For appropriate divisors $D \subset X$, the category of proper objects of $\Dwfuk(X,D)^\bc_p$ embeds into $\tilde{\fuk}\left(X,\omega^\C\right)$.
\end{itemize}
It's not clear exactly how to construct such a category -- it would likely involve incorporating additional objects supported on singular Lagrangians (e.g., skeleta of divisor complements).\footnote{Of course the idea to include singular Lagrangians in the Fukaya category is of interest in other contexts, for example in family Floer theory \cite{Fukaya2002b}.}
Whatever the construction, it would certainly have a payoff: by analogy with our previous discussion, $\tilde{\fuk}$ would have a better chance of `having enough objects to be well-behaved'; and the large volume/large complex structure limit HMS equivalences described in \S \ref{subsec:wrapcomp} could be combined with the versality techniques of \S \ref{sec:provehms} to prove mirror symmetry for compact Calabi--Yau varieties in the form
\[ \tilde{\fuk}\left(X,\omega^\C\right) \simeq \Dcoh\left(X^\circ_{\Psi(p)}\right).\]
For example, one could hope to prove HMS for many Batyrev mirrors via this approach, using the results of \cite{Gammage2017}.

\begin{rmk}
We suggest that $\Dfuk\left(X,\omega^\C\right)^\bc \subset \tilde{\fuk}\left(X,\omega^\C\right)$ would be a useful condition because it would allow us to address questions about symplectic topology of $(X,\omega)$ (for example, ruling out the existence of closed Lagrangian submanifolds with a given topology \`a la \cite{Abouzaid2010d}) without ever having to construct or compute with a non-trivial closed Lagrangian in $X$.
\end{rmk}

\begin{rmk}
\label{rmk:nofuksgiven}
If what one wants is to prove closed-string mirror symmetry, then in light of Remark \ref{rmk:nofuksneeded} it is not necessary to construct the category $\tilde{\fuk}$: instead, one would need to construct the cyclic open--closed map
\[ \tilde{\EuO\EuC}: HP_\bullet\left(\Dwfuk(X, D)^\bc\right) \to V^A(X)\]
and show it is a morphism of variations of Hodge structures. 
This is a much more tractable question, and in fact Ganatra has constructed the analogue of this map in the absence of the deformation by the Maurer--Cartan element \cite{Ganatraa}.
\end{rmk}


\bibliographystyle{amsalpha}
\bibliography{../../library}

\textsc{\small N. Sheridan, Centre for Mathematical Sciences, University of Cambridge, England.}\\
\textit{\small Email:} \texttt{N.Sheridan@dpmms.cam.ac.uk}\\

\end{document}